\documentclass[a4paper]{article}

\usepackage{lineno,hyperref}
\usepackage{amssymb,amsfonts,amsmath,amsthm}
\usepackage{mathrsfs}
\usepackage{here}
\usepackage[bbgreekl]{mathbbol}
\usepackage[dvipdfmx, hiresbb]{graphicx} 
\usepackage{type1cm}
\usepackage{anyfontsize}
\usepackage{enumitem}
\modulolinenumbers[5]

\newtheorem{definition}{Definition}[section]
\newtheorem{assumption}[definition]{Assumption}

\newtheorem{lemma}[definition]{Lemma}
\newtheorem{proposition}[definition]{Proposition}
\newtheorem{theorem}[definition]{Theorem}
\newtheorem{remark}[definition]{Remark}
\newtheorem{example}[definition]{Example}
\numberwithin{equation}{section}

\def\calB{{\cal B}}

\def\calE{{\cal E}}
\def\calF{{\cal F}}

\def\calI{{\cal I}}
\def\calJ{{\cal J}}
\def\calK{{\cal K}}

\def\calN{{\cal N}}

\def\calS{{\cal S}}
\def\calT{{\cal T}}

\def\calX{{\cal X}}

\def\simle{ \raisebox{-.7ex}{ $\stackrel{{\textstyle <}}{\sim}$ \em} }

\def\ep{\epsilon}

\def\up{\uparrow}

\def\Dup{D_{\up}(E, \real)}

\def\itemi{\item[\textnormal{(i)}]}
\def\itemii{\item[\textnormal{(ii)}]}
\def\itemiii{\item[\textnormal{(iii)}]}
\def\itemiv{\item[\textnormal{(iv)}]}
\def\itemv{\item[\textnormal{(v)}]}

\newcommand{\bbB}{{\mathbb B}}

\newcommand{\bbF}{{\mathbb F}}

\newcommand{\bbL}{{\mathbb L}}

\newcommand{\bbX}{{\mathbb X}}
\newcommand{\bbY}{{\mathbb Y}}
\newcommand{\bbZ}{{\mathbb Z}}

\newcommand{\bbx}{{\mathbb x}}
\newcommand{\bby}{{\mathbb y}}

\newcommand{\real}{\mathbb{R}}
\newcommand{\naturals}{\mathbb{N}}
\newcommand{\integer}{\mathbb{Z}}

\newcommand{\Stheta}{\hat{\theta}^{0, (q)}_T}
\newcommand{\Snu}{\hat{\nu}^{0, (q)}_T}

\DeclareMathOperator*{\diam}{diam}
\DeclareMathOperator*{\argmin}{argmin}
\DeclareMathOperator*{\argmax}{argmax}

\begin{document}

\title{Sparse estimation for generalized exponential marked Hawkes process}


\author{
	Masatoshi Goda\thanks{Graduate School of Mathematical Sciences, University of Tokyo: 3-8-1 Komaba, Meguro-ku, Tokyo 153-8914, Japan. E-mail: goda@ms.u-tokyo.ac.jp} \thanks{Japan Science and Technology, CREST, Japan.}
}

\maketitle


\abstract{We established a sparse estimation method for the generalized exponential marked Hawkes process by the penalized method to ordinary method (P-O) estimator.
Furthermore, we evaluated the probability of the correct variable selection.
In the course of this, we established a framework for a likelihood analysis and the P-O estimation when there might be nuisance parameters, and the true value of the parameter might be at the boundary of the parameter space.
Finally, numerical simulations are given for several important examples.}

\maketitle

\section{Introduction}
The Hawkes process is a self-exciting point process introduced by \cite{Hawkes1971} and has a wide range of applications including seismic (see \cite{Ogata1981}), finance (see \cite{AbergelEtal2016}), and web data analysis (see \cite{GodaEtal2021}).
For the properties as a stochastic process, a class with exponential kernels has attracted much attention since its intensity process has Markov, geometric ergodic, and mixing properties, for example, see \cite{ClinetYoshida2017} and \cite{Goda2021}.
As an extension of the Hawkes process, the marked Hawkes process is known.
By the marked Hawkes process, it is possible to consider a model that adds the scale of the event and other characteristics to the occurrence times of the events.
The properties of the marked Hawkes process were investigated by \cite{Clinet2021}, and in which an important class called the generalized exponential marked Hawkes process (GEMHP) was introduced.
The GEMHP is a marked Hawkes process whose kernel has the flexible form, which is generated by terms multiplying $e^{-r \cdot}$, where $r>0$, by a polynomial or trigonometric function.
The GEMHP can be represented in terms of a Markov process, and we can establish the geometric ergodicity of the GEMHP.
Moreover, the convergence of moments for the quasi maximum likelihood estimator (QMLE) and the quasi Bayesian estimator (QBE) was established, in \cite{Clinet2021}, for a class of the GEMHP with a linear intensity.
In particular, a polynomial type large deviation inequality was established for the quasi-likelihood of the GEMHP in the application of the results of the quasi-likelihood analysis in \cite{Yoshida2011}.
The GEMHP with a linear intensity has been used for the models of earthquakes marked with the magnitude (see \cite{Ogata1981}), the limit order book in financ (see \cite{RambaldEtal2017}), etc.
In this paper, we introduce the Hawkes process marked by "topic" as an example of the GEMHP with a linear intensity.

\

Model selection is one of the most important topics in statistical inference.
In particular, sparse estimation methods like the least absolute shrinkage and selection operator (LASSO) in \cite{Tibshirani1996}, the elastic net in \cite{ZouHastie2005}, and so on, have been widely studied starting from linear regression problems.
By the sparse estimation method, we can execute parameter estimation and variable selection simultaneously.
Let $\theta^* = (\theta^*_j)_{j=1,\dots,p}$ be the true parameter of a parameter $\theta = (\theta_j)_{j=1,\dots,p}$ for $p \in \naturals$.
Moreover, let $\calJ^0 = \{ j \arrowvert \theta^*_j = 0\}$ and $\calJ^1 = \{ j \arrowvert \theta^*_j \neq 0\}$.
We write $\theta_{\calJ} = (\theta_j)_{j \in \calJ}$ for a vector $\theta$ and an index set $\calJ$.
The following two properties are called the oracle properties (see \cite{FanLi2001}) that the sparse estimator $\hat{\theta}_T$ should satisfy:
\begin{itemize}
  \item Selection consistency: $P\big[\hat{\theta}_{\calJ^0} = 0\big] \to 1$,
	\item Asymptotic normality: $\sqrt{T}\big(\hat{\theta}_{\calJ^1} - \theta^*_{\calJ^1}\big) \to^d N(0, \Gamma^{-1})$,
\end{itemize}
as $T \to \infty$ for some positive definite matrix $\Gamma$, where $T$ represents an observation time, and we omit $T$ in the expression $\hat{\theta}_{\calJ^0}$ and $\hat{\theta}_{\calJ^1}$.
We focus on the penalized method to ordinary method (P-O) estimator proposed in \cite{SuzukiYoshida2020}.
The P-O estimation is a sparse estimation method using the least-squares approximation method given a prior estimator with the consistency, and retuning using an ordinary estimation method such as the maximum likelihood estimator.
It satisfies the oracle properties under the suitable conditions.
Moreover, we can evaluate the probability of the correct variable selection.

\

A sparse estimation for the multivariate Hawkes process is useful to identify disconnections in networks.
It is also possible to identify which marks do not affect the trend by a sparse estimation for the marked Hawkes process.
Furthermore, it is also important that a sparse estimation prevents the overfitting of the model.
In a previous study applying a sparse estimation to the Hawkes process, \cite{HansenEtal2015} proposed an adaptive $L^1$-penalized methodology for the nonparametric case and evaluated the oracle inequality.
In the parametric case, there are several studies with respect to the Hawkes process using the exponential kernel.
In \cite{ZhouEtal2013}, they designed the log-likelihood function penalized by the nuclear and $L^1$ norm and an algorithm ADM4 for their estimation method, while they investigated the performance of their method through numerical experiments.
In \cite{BacryEtal2020}, they proposed the least-squares method with the entry-wise weighted nuclear and $L^1$ norm penalization and proved a sharp oracle inequality for their procedure.
In \cite{GodaEtal2021}, they introduced a hybrid method combined with the QMLE and the $L^1$-penalized QMLE and investigated the accuracy of model selection and the asymptotic normality through numerical experiments.
However, the oracle properties have not been established even for the exponential Hawkes process with no marks.

We apply the P-O estimation to the GEMHP, and we prove the oracle properties and evaluate the probability of the correct variable selection.
In this application, when the GEMHP contains zero parameters, it is often necessary to assume the existence of nuisance parameters.
For example, let $N_t = (N^1_t, \dots, N^d_t)$ be a $d$-dimensional exponential Hawkes process with the intensity process
\begin{eqnarray}
\label{Introduction eq1}
	\lambda^i_t = \mu_i + \sum_{j=1}^d \int_0^{t-} \alpha_{ij} e^{-\beta_{ij}(t-s)} dN^j_t, \ i= 1, \dots, d,
\end{eqnarray}
where $\mu_i$, $\alpha_{ij}$, and $\beta_{ij}$ are parameters for $i, j= 1, \dots, d$.
Then, $\beta_{ij}$ is undefined, that is, $\beta_{ij}$ is a nuisance parameter, when $\alpha_{ij}=0$.
We confirm that the polynomial type large deviation inequality for the quasi-likelihood of the GEMHP with nuisance parameters holds and thus that the consistency for the QMLE and the QBE holds.
Furthermore, in most cases of statistical inference, the true value of a parameter is represented as an interior point on a compact set.
However, when the true value of a parameter in the GEMHP is zero, it might be at the boundary of the parameter space.
For an example of an exponential Hawkes process with the intensity (\ref{Introduction eq1}), $\alpha_{ij}$'s are often assumed to take value in a compact subset in $[0, \infty)$.
Then, $\alpha_{ij}=0$ is realized on the boundary of the parameter space.
We prove that the P-O estimator works well even in such a situation where the true value is on the boundary.

\

We explain the GEMHP and its properties in Section \ref{GEMHP}.
In Section \ref{P-O}, we discuss the P-O estimator under the condition that some parameters are nuisance parameters and the true parameter is possibly on the boundary of the parameter space.
The main results about the application of the P-O estimation to the sparse GEMHP are in Section \ref{App to GEMHP}.
Finally, Section \ref{Ex and Sim} presents the results of some numerical experiments.
We introduce the Hawkes process marked by "topic" in this section.
The proofs of each statement are given in Appendix \ref{Appendix Proofs}.
Moreover, we give Additional numerical experiments in Appendix \ref{Appendix Sim}.


\section{Generalized exponential marked Hawkes process}
\label{GEMHP}

In this section, we review the theory in \cite{Clinet2021}.
In particular, we define the generalized exponential marked Hawkes process (GEMHP).
The GEMHP is a class of marked Hawkes processes that satisfies the Markov, ergodic, and mixing properties under the stability conditions.
In this article, we only focus on the GEMHP with the linear form intensity for our application in Section \ref{App to GEMHP}. 
On the other hand, we note that the Markov, ergodic, and mixing properties are proved for more general non-linear (sub-linear) form intensities.

\subsection{Marked point process}
\label{sub Marked point process}

First, we define the general marked point process.
Let $\bbB = (\Omega, \calF, \bbF=\{\calF_t\}_{t \ge 0}, P)$ be a stochastic basis, and $(\bbX, \calX)$ be a measurable space.
For $d \in \naturals$, we consider a sequence of couples $(T^i_n, X^i_n)_{n \in \integer, i = 1, \dots, d}$.
Suppose that $T^i_n$'s are $\bbF$-stopping times such that almost surely $T^i_0 = 0 < T^i_1 < \cdots < T^i_n < \cdots < \infty$ and $T^i_n \to \infty$ as $n \to \infty$ hold for each $i$, and $X^i_n$'s are $\bbX$-valued $\calF_{T^i_n}$-measurable random variables.
We define the $d$-dimensional marked point process $\bar{N} = (\bar{N}^1, \dots, \bar{N}^d)$ as a family of random measures on $\real_+ \times \bbX$ such that $\bar{N}^i(ds, dx) = \sum_{n \in \naturals} \delta_{(T^i_n, X^i_n)}(ds, dx)$.
Moreover, we call a random measure $\nu^i(ds, dx)$ the compensator of $\bar{N}^i(ds, dx)$ when $\bar{N}^i([0,t] \times F) - \nu^i([0,t] \times F)$ is a local martingale for any $F \in \calX$, see Theorem 1.8 in \cite{JacodShiryaev2003}.

\subsection{Generalized exponential marked Hawkes process}
\label{sub GEMHP}

We use the same notations in Subsection \ref{sub Marked point process}.
We write the counting process associated with $\bar{N}^i$ as $N^i_t = \bar{N}([0,t] \times \bbX)$ and jump times of the global counting process $\sum_{i=1}^d N^i_t$ as $(T_n)_{n \in \naturals}$.
Furthermore, let $(X_n)_{n \in \naturals}$ be a permutation of $(X^i_n)_{n \in \naturals, i = 1, \dots, d}$ similar to the relationship between jump times $(T_n)_{n \in \naturals}$ and $(T^i_n)_{n \in \naturals, i = 1, \dots, d}$.
Then, we define the mark process $X_t$ as a piecewise constant and right continuous stochastic process such that $X_t = X_n$ for $t \in [T_n, T_{n+1})$ where $T_0 = 0$ and $X_0 = X^1_0$.
The marked Hawkes process is defined as below.
\begin{definition}
\label{GEMHP Def 1}
A $d$-dimensional marked point process $\bar{N}$ is called a $d$-dimensional marked Hawkes process if the intensity process of the associated counting process $N^i$ has the form
\[
	\lambda^i_t = \phi_i \left( \left( \int_{[0,t) \times \bbX} h_{ij} (t-s, x) \bar{N}^j(ds, dx)\right)_{j = 1, \dots, d} , X_{t-} \right), \ i = 1, \dots, d,
\]
where $\phi_i \colon \real^d_+ \times \bbX \to \real_+$ is a continuous function and $h_{ij} \colon \real_+ \times \bbX \to \real_+$ is a measurable function for each $i , j = 1, \dots, d$.
\end{definition}

Then, the linear GEMHP is defined by restricting the form of the function $\phi_i$ and the kernel function $h_{ij}$.
For $p \in \naturals$, we denote the Frobenius inner product on a real matrix space $\real^{p \times p}$ as $\langle \cdot \arrowvert \cdot \rangle$.
\begin{definition}
\label{GEMHP Def 2}
A $d$-dimensional marked Hawkes process $\bar{N}$ is called a $d$-dimensional linear generalized exponential marked Hawkes process if the function $\phi_i$ and the kernel function $h_{ij}$ have the representations, for $i, j = 1, \dots, d$,
\[
	\phi_i(u, x) = \mu_i(x) + \sum_{j=1}^d u_j
\]
and
\[
	h_{ij}(s, x) = \langle A_{ij} \arrowvert e^{-sB_{ij}} \rangle g_{ij}(x),
\]
where $\mu_i \colon \bbX \to \real_+$ and $g_{ij} \colon \bbX \to \real_+$ are measurable functions, $A_{ij}, B_{ij} \in \real^{p \times p}$ for some $p \in \naturals$, and $e^{-sB_{ij}}$ is the matrix exponential for each $i , j = 1, \dots, d$.
That is, its intensity process has the form
\[
	\lambda^i_t = \mu_i(X_{t-}) + \sum_{j=1}^d \int_{[0,t) \times \bbX} \langle A_{ij} \arrowvert e^{-sB_{ij}}\rangle g_{ij}(x) \bar{N}^j(ds, dx)
\]
for $i = 1, \dots, d$.
\end{definition}
The temporal part $\langle A \arrowvert e^{-sB} \rangle$ of the kernel is represented as a linear combination of terms $P(s)(1 + C_1\cos(\xi s) + C_2\sin(\xi s))e^{-r}$, where $\xi, C_1, C_2 \in \real$ and $r>0$ are constants, and $P(s)=\sum_{k=0}^Pa_ks^k$ is a polynomial for $a_0,\dots, a_P \in \real$ and some $P \in \naturals$, see Proposition 3.1 of \cite{Clinet2021}.

The remainder of this section is devoted to the explanation of the sufficient conditions of Theorem \ref{GEMHP thm 1}.
We restrict the distribution of the mark process $X_t$ to maintain the Markov structure of the intensity process.
Let $(\kappa_n)_{n \in \naturals}$ be labels of the jumps of the global counting process $\sum_{i=1}^d N^i$, i.e., $\kappa_n$ is a $\{1, \dots, d\}$-valued random variable such that $\Delta N^{\kappa_n}_{T_n} = 1$ for $n \in \naturals$.
We write $\Delta T_n = T_n - T_{n-1}$.
Then, we assume that there exists a family of Feller transition kernels $\{Q_i\}_{i=1, \dots, d}$ on $\bbX \times \calX$ such that
\begin{eqnarray}
\label{GEMHP eq 1}
	P \big[ X_n \in F \big\arrowvert \kappa_n, \Delta T_n, \calF_{T_{n-1}}\big] = Q_{\kappa_n}(X_{n-1}, F)
\end{eqnarray}
for any $F \in \calX$ and $n \in \naturals$.
For the stability of the process, we assume that $B_{ij}$ has eigenvalues with positive real parts.
We define $\Phi_{ij}(x) = \int_{\real_+ \times \bbX} h_{ij}(s, y) Q_j(x, dy) ds$ and $G_{ij}(x) = \int_{\bbX} g_{ij}(y)Q_j(x, dy)$ for $x \in \bbX$.
$\Phi(x) = \{ \Phi_{ij}(x) \}_{i, j = 1, \dots, d}$ refers to the conditional expectation of the long-run effect of the excitation on the intensity process after a jump.
$G_{ij}(x)$ refers to the conditional expectation of a jump size in $\lambda^i$ when $N^j$ jumps.
By the assumption on $B_{ij}$, we have the representation $\Phi_{ij}(x) = \langle A_{ij} \arrowvert B^{-1}_{ij}\rangle G_{ij}(x)$ for any $i, j = 1, \dots, d$.

Let
\begin{eqnarray}
\label{GEMHP Eq 1}
	\calE^{ij}_t = \int_{[0, t] \times \bbX} e^{-(t-s)B_{ij}} g_{ij}(x) \bar{N}_j (ds, dx)
\end{eqnarray}
for $i, j = 1, \dots, d$ and $Z_t = (\calE_t, X_t) = \big((\calE^{ij}_t)_{i,j = 1, \dots, d}, X_t \big)$.
$Z_t$ obviously drives the intensity process of the GEMHP.
We write the transition kernel of the global mark process $X$ by $Q$.
It satisfies
\[
	Q(Z_{T_{i-1}}, \cdot) = \frac{1}{\xi(\Delta T_i, Z_{T_{i-1}})}\sum_{i=1}^d \xi^i(\Delta T_i, Z_{T_{i-1}})Q_i(X_{i-1}, \cdot),
\]
where $\xi(t, z) = \sum_{i=1}^d \xi^i(t, z)$ and $\xi^i(t, z) = \mu_i(x) + \sum_{j=1}^d \langle A_{ij} \arrowvert e^{-tB_{ij}} \ep_{ij} \rangle$ for $t \ge 0$ and $z = ((\ep_{ij})_{i,j}, x)$, see Proposition 3.2 of \cite{Clinet2021}.
We call a non-negative function $f$ a norm-like function if $f(x) \to \infty$ as $\lvert x \rvert \to \infty$ holds.
The following statements are the sufficient conditions for the geometric ergodicity and the geometric mixing property of the GEMHP.
\begin{description}
	\item[\textnormal{[L1]}] There exist norm-like functions $f_X$ and $u_X$ such that
		\begin{eqnarray}
		\label{GEMHP eq 2}
				\sum_{i=1}^d \mu_i(x) = O(u_X(x)) \text{ \ as $\lvert x \rvert \to \infty$, }
		\end{eqnarray}
		\begin{eqnarray}
		\label{GEMHP eq 3}
				\sum_{i=1}^d \mu_i(x) \int_{\bbX} \big\{ f_X(y) - f_X(x) \big\} Q_i(x, dy) \le -u_X(x) \text{ \ for any $x \in \bbX$, }
		\end{eqnarray}
	 	\begin{eqnarray}
		\label{GEMHP eq 4}
				\sum_{i=1}^d \mu_i(x) \sum_{j=1}^d G_{ji}(x) = o(u_X(x)) \text{ \ as $\lvert x \rvert \to \infty$, }
		\end{eqnarray}
		and there exists $\bar{c} > 0$ such that
		\begin{eqnarray}
		\label{GEMHP eq 5}
				\sup_{x \in \bbX, j =1, \dots, d} \int_{\bbX} e^{ \bar{c} \left[ \sum_{i=1}^d g_{ij}(y) + f_X(y) - f_X(x) \right] } Q_j(x, dy) < \infty.
		\end{eqnarray}
	\item[\textnormal{[L2]}] There exist $\kappa \in \real^d$ with positive coefficients and $\rho \in [0, 1)$ such that, component-wise, $\sup_{x\in \bbX} ( \Phi(x)^T \kappa) \le \rho \kappa$.
	\item[\textnormal{[ND1]}] There exist $\underline{\phi}, \underline{g} > 0$ such that $\phi_i > \underline{\phi}$ and $g_{ij} > \underline{g}$ for any $i, j = 1, \dots, d$.
	\item[\textnormal{[ND2]}] The transition kernel $Q$ admits a reachable point $x_0 \in \bbX$.  Moreover, for any $j = 1, \dots, d$, the transition kernel $Q_j$ admits a sub-component $\calT_j$, such that there exist a lower semi-continuous function $r_j \colon \bbX^2 \to \real_+$ and a non-trivial measure $\sigma_j$ on $\calX$, such that
	\begin{itemize}
		\item[$\bullet$] $\sigma_j(O) > 0$ for any non-empty open set $O \in \calX$,
		\item[$\bullet$] $\calT_j(x, F) = \int_{F} r_j(x, y) \sigma_j(dy)$ and $\calT_j(x, \bbX)>0$ for any $x \in \bbX$ and $F \in \calX$.
	\end{itemize}
\end{description}

We write the transition kernel of $Z$ as $P^t$ for $t \ge 0$.
Finally, the $V$-norm of a measure $\mu$ on a measurable space $(S, \calS)$ is defined as
\[
	\big\| \mu \big\|_V = \sup_{\psi \le V} \left\lvert \int_S \psi(s) \mu(ds) \right\rvert
\]
for a positive function $V$, where the supremum is taken over all the measurable functions $\psi$ such that $\psi(s) \le V(s)$ for all $s \in S$.
The following theorem is the main theorem in this section, which is shown in Theorem 3.7 of \cite{Clinet2021}.

\begin{theorem}{\textnormal{(Theorem 3.7 in \cite{Clinet2021})}}
\label{GEMHP thm 1}
Under [L1]-[L2] and [ND1]-[ND2], $Z$ is $V$-geometrically ergodic, i.e., there exist a unique invariant measure $\pi$ and constants $(a_{ij})_{i,j=1, \dots, d} \in \real^{p^2d^2}, \eta >0, C \ge 0, r \in [0, 1)$ such that for any $t>0$ and $z = (\ep, x) \in \real^{p^2d^2} \times \bbX$
\[
	\big\| P^t(z, \cdot) - \pi \big\|_V \le C ( 1 + V(z) ) r^t
\]
where $V(z) = \exp\big( \sum_{i, j = 1}^d \langle a_{ij} \arrowvert \ep_{ij} \rangle + \eta f_X(x) \big)$.
Moreover, $Z$ is $V$-geometrically mixing, i.e., there exist positive constants $C'>0$, $r' \in [0,1)$ such that for any $t,u \ge 0$ and measurable functions $\phi, \psi$ with $\phi^2 \le V, \psi^2 \le V$,
\[
	\Big\lvert E\big[\phi(Z_{t+u})\psi(Z_t) \big\arrowvert Z_0=z\big] - E\big[\phi(Z_{t+u})\big\arrowvert Z_0=z\big]E\big[\psi(Z_t)\big\arrowvert Z_0=z\big]\Big\rvert \le C'V(z)r'^u.
\]
\end{theorem}


\section{P-O estimator with nuisance parameter}
\label{P-O}

The penalized method to ordinary method (P-O) estimator was introduced by \cite{SuzukiYoshida2020}.
The P-O estimator allows us to execute parameter estimation and variable selection simultaneously.
In this section, we extend the applicable condition of the P-O estimation to the case where nuisance parameters might exist, and the true parameter might be at the boundary of the parameter space.
In general, a parameter that is not subject to estimation is called a nuisance parameter.
In particular, we admit a nuisance parameter whose true value might be undefined.

Let $\vartheta = (\theta^0, \theta^1, \nu^0, \nu^1) \in \Xi = \bar{\Theta}^0 \times \Theta^1 \times \bar{\calN}^0 \times \calN^1$ be a parameter, where $\Theta^0, \Theta^1, \calN^0$, and $\calN^1$ are open convex bounded subsets of $\real^{p_0}, \real^{p_1}, \real^{n_0}$, and $\real^{n_1}$ for $p_0, p_1, n_0, n_1 \in \naturals$, respectively.
$\theta^0 = (\theta^0_1, \dots, \theta^0_{p_0})$ is a parameter that might take the value $\theta^0_i = 0$ for some $i$.
$\theta^1 = (\theta^1_1, \dots, \theta^1_{p_1})$ is a non-zero parameter.
$\nu^0 = (\nu^0_1, \dots, \nu^0_{n_0})$ is a nuisance parameter that might take the value $\nu^0_i = 0$ for some $i$.
$\nu^1 = (\nu^1_1, \dots, \nu^1_{n_1})$ is a nuisance parameter that takes a non-zero value.
Note that $\nu^0$ and $\nu^1$ might not have the true value.
Suppose that $(\theta^{0*},  \theta^{1*}) = (\theta^{0*}_1, \dots, \theta^{0*}_{p_0}, \theta^{1*}_1, \dots, \theta^{1*}_{p_1}) \in \bar{\Theta}^0 \times \Theta^1$ is the true parameter, and let $\calJ^0 = \big\{ j = 1, \dots, p_0 \big\arrowvert \theta^{0*}_j = 0 \big\}$ and $\calJ^1 = \big\{ j = 1, \dots, p_0 \big\arrowvert \theta^{0*}_j \neq 0 \big\}$.
We write $d_1 = \# \calJ^1$.
We remark that $\theta^{0*}$ might be on the boundary of $\Theta_0$.

\begin{example}
Let $\bar{N}$ be a GEMHP with the intensity $\lambda^i_t(\vartheta^*)$: Here, for $i = 1, \dots, d$,
\begin{eqnarray*}
	\lambda^i_t(\vartheta) = \mu_i + \sum_{j=1}^d \int_{[0,t) \times \real^{d'}} \Bigg( \sum_{k=0}^p \alpha_{ijk}s^k \Bigg) \Bigg(1 + \sum_{l=1}^{d'} m_{ijl}x_{l}^2 \Bigg) e^{-(t-s)\beta_{ij}} \bar{N}^j(ds, dx),
\end{eqnarray*}
where $\vartheta = \big( (\mu_i)_i, (\alpha_{ijk})_{i,j,k}, (\beta_{ij})_{i,j}, (m_{ijl})_{i,j,l} \big)$ are non-negative parameters and $\vartheta^* = \big( (\mu_i^*)_i, (\alpha_{ijk}^*)_{i,j,k},$ $(\beta_{ij}^*)_{i,j},$ $(m_{ijl}^*)_{i,j,l} \big)$ is the true parameters with nuisance parameters.
We can consider the situation where some $\alpha_{ijk}$ and $m_{ijl}$ might be $0$ besides $\mu_i$ and $\beta_{ij}$ are always positive.
In this case, $\beta_{ij}^*$ and $m_{ijl}^*$ are undefined parameters, and thus nuisance parameters, when $\alpha^*_{ijk}=0$ for all $k$.
In other words, we can write $\theta^0 = \big((\alpha_{ijk})_{i,j,k}, (m_{ijl} \arrowvert i,j,l \ s.t. \ (i,j) \notin \calI^0)\big)$, $\theta^1 = \big((\mu_i)_i, (\beta_{ij} \arrowvert i,j \ s.t. \ (i,j) \notin \calI^0)\big)$, $\nu^0 = (m_{ijl} \arrowvert i,j,l \ s.t. \ (i,j) \in \calI^0)$, and $\nu^1 = (\beta_{ij} \arrowvert i,j \ s.t. \ (i,j) \in \calI^0)$, where $\calI^0 = \big\{(i,j) \big\arrowvert \alpha^*_{ijk} = 0 \text{ for all } k = 1, \dots, p \big\}$.
Then, we can rewrite $\vartheta = (\theta^0, \theta^1, \nu^0, \nu^1)$.
\end{example}

Let $T>0$ be an observation time index.
We often consider the case where $T \in \naturals$ with discrete time observation or $T \in \real_+$ with continuous time observation.
For convenience of explanation, we consider two objective functions $\bbL^1_T \colon \Xi \to \real$ and $\bbL^2_T \colon \Xi \to \real$.
For example, $\bbL^1_T$ and $\bbL^2_T$ are log-likelihood functions multiplied by $-1$.
The objective function $Q^{(q)}_T \colon \bar{\Theta}^0 \times \bar{\calN}^0 \to \real$ will be defined later.
We denote the $j$-th component of each estimator $\theta_T$ as $\theta_j$ by omitting $T$.
If $\argmin$ is realized at multiple points, we take one of them arbitrarily.
The P-O estimation is done in the following three steps:

\begin{description}
	\item[{\bf Step 1.}] We obtain the first estimator of $\vartheta$ by
		\[
			\big( \tilde{\theta}^0_T, \tilde{\theta}^1_T, \tilde{\nu}^0_T, \tilde{\nu}^1_T \big) = \argmin_{\vartheta \in \Xi} \bbL^1_T(\vartheta).
		\]
	\item[{\bf Step 2.}] We obtain the second estimator of $(\theta^0, \nu^0)$ by
		\[
			\big( \Stheta, \Snu \big) = \argmin_{ (\theta^0, \nu^0 ) \in \bar{\Theta}^0 \times \bar{\calN}^0} Q^{(q)}_T(\theta^0, \nu^0),
		\]
		where the objective function $Q^{(q)}_T$ depends on the first estimator $\big( \tilde{\theta}^0_T, \tilde{\nu}^0_T \big)$.
	\item[{\bf Step 3.}] We obtain the third estimator of $\vartheta$ by
		\[
			\big( \check{\theta}^0_T, \check{\theta}^1_T, \check{\nu}^0_T, \check{\nu}^1_T \big) = \argmin_{ \vartheta \in \hat{\Theta}^0_T \times \Theta^1 \times \hat{\calN}^0_T \times \calN^1} \bbL^2_T(\vartheta),
		\]
		where $\hat{\Theta}^0_T = \big\{\theta^0 \in \bar{\Theta}^0 \big\arrowvert \theta^0_j = 0, j \in \hat{\calJ}^0_T \big\}$, $\hat{\calJ}^0_T = \big\{ j = 1, \dots, p_0 \big\arrowvert \hat{\theta}^{0, (q)}_j = 0 \big\}$, $\hat{\calN}^0_T = \big\{\nu^0 \in \bar{\calN}^0 \big\arrowvert \nu^0_j = 0, j \in \hat{\calK}^0_T \big\}$, and $\hat{\calK}^0_T = \big\{ j = 1, \dots, n_0 \big\arrowvert \hat{\nu}^{0, (q)}_j = 0 \big\}$.
\end{description}
We call $\check{\vartheta}_T = \big( \check{\theta}^0_T, \check{\theta}^1_T, \check{\nu}^0_T, \check{\nu}^1_T \big)$ the P-O estimator.
The sparse model selection is derived from Step 2.
The objective function $Q^{(q)}_T(\theta^0, \nu^0)$ is constructed by using the first estimators $\tilde{\theta}^0_T, \tilde{\nu}^0_T$ as below:
{\small\begin{eqnarray}
\label{P-O Eq 2-1}
	Q^{(q)}_T(\theta^0, \nu^0) = \sum_{j=1}^{p_0} \left( \big( \theta^0_j - \tilde{\theta}^0_j \big)^2 + \kappa^{\theta}_j \lvert \theta^0_j \rvert^q \right) +  \sum_{j=1}^{n_0} \left( \big( \nu^0_j - \tilde{\nu}^0_j \big)^2 +\kappa^{\nu}_j  \lvert \nu^0_j \rvert^q \right),
\end{eqnarray}}
\hspace{-5pt}where $q \in (0, 1], \kappa^{\theta}_j = \alpha_T \big\lvert \ep_T + \tilde{\theta}^0_j \big\rvert^{-\gamma},  \kappa^{\nu}_j = \alpha_T \big\lvert \ep_T + \tilde{\nu}^0_j \big\rvert^{-\gamma}, \gamma > -(1-q)$, and $\ep_T, \alpha_T$ are deterministic sequences.
We remark that $\kappa^{\theta}_j$ and $\kappa^{\nu}_j$ depend on $T$.
We often choose $\ep_T$ and $\alpha_T$ to converge to $0$ as $T \to \infty$, see Remark \ref{P-O Rem 2}.
Let $a_T = \max_{ j \in \calJ^1} \kappa^{\theta}_j$ and $b_T = \min_{j \in \calJ^0} \kappa^{\theta}_j$.

We retune the estimator in Step 3.
We rewrite the notation in Step 3 to describe the properties of components with the non-zero true values in the third estimator.
Without loss of generality, we can assume that $\theta^0 = (\phi, \psi) \in \bar{\Theta}^0_{\phi} \times \bar{\Theta}^0_{\psi} = \bar{\Theta}^0$ and its true value $\theta^{0*} = (\phi^*,\psi^*) = (\phi^*, 0)$.
Let $\bar{\bbL}^2_T(\phi, \theta^1, \nu^0, \nu^1) = \bbL^2_T(\phi, 0, \theta^1, \nu^0, \nu^1)$ and
\[
	\big( \bar{\phi}_T, \bar{\theta}^1_T, \bar{\nu}^0_T, \bar{\nu}^1_T \big) = \argmin_{ (\phi, \theta^1, \nu^0, \nu^1) \in \bar{\Theta}^0_{\phi} \times \Theta^1 \times \bar{\calN}^0 \times \calN^1 } \bar{\bbL}^2_T(\phi, \theta^1, \nu^0, \nu^1).
\]
By contrary, we write $\check{\theta}^0_{\calJ^1} = \big(\check{\theta}^{0}_j\big)_{j \in \calJ^1}$.
We call a stochastic process $X_T$ is $L^{\infty-}$-bounded if $\sup_T E\left[ \left\lvert  X_T\right\lvert ^p \right] < \infty$ holds for all $p \ge 1$.
We introduce sufficient conditions for the oracle properties of the above P-O estimator.
Note that the following assumptions demand the properties of the first and third-step estimators.
In other words, it is not necessary to assume the objective functions $\bbL^1_T$ and $\bbL^2_T$ to establish the oracle properties of the P-O estimator.

\begin{assumption}
\label{P-O Asm 1}
\begin{enumerate}
	\itemi $\tilde{\theta}^0_T$ is $\sqrt{T}$-consistent, that is, $\sqrt{T} \big( \tilde{\theta}^0_T - \theta^{0*} \big) = O_p(1)$ as $T \to \infty$.
	\item[(i)$'$] $\sqrt{T} \big( \tilde{\theta}^0_T - \theta^{0*} \big)$ is $L^{\infty-}$-bounded.
	\itemii $T^{1/2} a_T = O_p(1)$ and $T^{(2-q)/2} b_T \to^p \infty$ as $T \to \infty$.
	\itemiii $\sqrt{T}\big\{ ( \bar{\phi}_T, \bar{\theta}^1_T ) - ( \phi^*, \theta^{1*} ) \big\} \to^d \Lambda^{-1/2}(\nu^*) \xi$ as $T \to \infty$, where $\Lambda(\nu^*) \in \real^{(d_1+ p_1) \times (d_1+ p_1)}$ is a positive definite matrix for any $\nu^* \in \bar{\calN}^0 \times \calN^1$, and $\xi$ is a $(d_1+ p_1)$-dimensional standard Gaussian random vector.
	\itemiv There exists $\ep \in (-1+q, \gamma)$ such that $T^{-(1+\gamma-\ep)/2} \alpha_T^{-1} = O(1)$ as $T \to \infty$.
	\itemv $\sqrt{T}\big\{ ( \bar{\phi}_T, \bar{\theta}^1_T ) - ( \phi^*, \theta^{1*} ) \big\}$ is $L^{\infty-}$-bounded.
\end{enumerate}
\end{assumption}

\begin{remark}
\label{P-O Rem 1}
	Assumption \ref{P-O Asm 1} (i)$'$ is obviously a stronger condition than (i).
	We need (i)' to evaluate the probability of a correct variable selection for the P-O estimator.
	In the case where $\bbL_T^1$ and $\bbL_T^2$ are the quasi log-likelihood functions, (i)' and (v) are satisfied if the polynomial type large deviation inequality for the quasi likelihood ratio random field holds, see Proposition 1 in \cite{Yoshida2011}.
\end{remark}

\begin{remark}
\label{P-O Rem 2}
Assumptions \ref{P-O Asm 1} (ii) and (iv) are conditions on the weight of the penalty term.
These conditions are satisfied, for example, in the following setup under Assumption \ref{P-O Asm 1} (i).
We take $a \in (0, 1-q+\gamma)$ for given $q \in (0, 1]$ and $\gamma > -1+q$.
Moreover, we take positive deterministic sequences $\alpha_T, \ep_T$ that satisfy $\alpha_T = O\big(T^{-(1+a)/2})$ and $\ep_T = O(T^{-1/2})$.
Then, as $T \to \infty$,
\[
	T^{\frac{1}{2}} a_T = O\big(T^{-\frac{a}{2}}\big) \left\lvert  O\big(T^{-\frac{1}{2}}\big) + \min_{j \in \calJ^1} \big(\tilde{\theta}^0_j -\theta^{0*}_j \big) + \min_{j \in \calJ^1 } \theta^{0*}_j  \right\lvert ^{-\gamma} = o_p(1)
\]
and
\[
T^{\frac{2-q}{2}} b_T= O\big(T^{\frac{1-q+\gamma-a}{2}}\big) \left\lvert  O(1) + \max_{j \in \calJ^0} \sqrt{T}\big(\tilde{\theta}^0_j -\theta^{0*}_j \big)  \right\lvert ^{-\gamma} \to^p \infty.
\]
Moreover, $T^{-(1+\gamma-\ep)/2} \alpha_T^{-1} = O(1)$ as $T \to \infty$ holds for $\ep = \gamma - a \in (-1+q, \gamma)$.
\end{remark}

The following theorem is the main theorem in this section.
The statement (i) means the selection consistency in Step 2.
Moreover, the statements (ii) and (iv) say that the probability of a correct variable selection converges to $1$ in a polynomial decay as $T \to \infty$.
The statement (iii) is about the asymptotic normality of the P-O estimator.
In particular, the statements (iii) and (iv) mean the oracle properties of the P-O estimator.
Compared to the result of \cite {SuzukiYoshida2020}, each statement is extended to in the situation where nuisance parameters might exist and the true parameter might be at the boundary of the parameter space.
\begin{theorem}
\label{P-O Thm 1}
\begin{enumerate}
	\itemi Under Assumption \ref{P-O Asm 1} (i) and (ii),
	\[
		P\left[ \hat{\calJ}^0_T = \calJ^0 \right] \to 1 \ \text{ as \ $T \to \infty$}.
	\]
	\itemii Under Assumption \ref{P-O Asm 1} (i)', (ii) and (iv), for all $L>0$, there exists $C_L >0$ such that
	\[
		P\left[ \hat{\calJ}^0_T = \calJ^0 \right] \ge 1 - C_L T^{-L}
	\]
	for all $T>0$.
	\itemiii Under Assumption \ref{P-O Asm 1} (i), (ii) and (iii)
	\[
		\sqrt{T}\big\{ ( \check{\theta}^0_{\calJ^1}, \check{\theta}^1_T ) - ( \phi^*, \theta^{1*} ) \big\} \to^d \Lambda^{-1/2}(\nu^*) \xi \ \text{ as \ $T \to \infty$}.
	\]
	\itemiv Under Assumption \ref{P-O Asm 1} (i)$'$, (ii), (iv) and (v), $\sqrt{T} \big\{ (\check{\theta}^0_T, \check{\theta}^1_T ) - (\theta^{0*}, \theta^{1*}) \big\}$ is $L^{\infty-}$-bounded. Moreover, for all $L>0$, there exists $C_L >0$ such that
	\[
		P\left[ \check{\calJ}^0_T = \calJ^0 \right] \ge 1 - C_L T^{-L}
	\]
	for all $T>0$, where $\check{\calJ}^0_T = \big\{ j = 1, \dots, p_0 \big\arrowvert \check{\theta}^0_j = 0 \big\}$.
\end{enumerate}
\end{theorem}


\section{Application to GEMHP}
\label{App to GEMHP}
In this section, we apply the penalized method to ordinary method (P-O) estimator to the generalized exponential marked Hawkes process (GEMHP).

For tractability, we restrict the intensity of the GEMHP to a linear form, that is, $\phi_i(u,x) = \mu_i(x) + \sum_{j=1}^d u_j$ for any $i=1,\dots,d$.
When $\phi_i$ is sub-linear, more involved formulations for the conditions [AH1]-[AH2] below are needed, and we let it aside for future works as in Section 4 of \cite{Clinet2021}.
We set a parameter $\vartheta = (\theta^0, \theta^1, \nu^0, \nu^1) \in \Xi = \bar{\Theta}^0 \times \Theta^1 \times \bar{\calN}^0 \times \calN^1$ as in Section \ref{P-O}.
For some $\vartheta^* \in \Xi$, we assume that $\bar{N}$ is a $d$-dimensional GEMHP with the following intensity process:
{\small\begin{eqnarray}
\label{App to GEMHP Eq 1}
	\hspace{-7pt}\lambda^i_t(\vartheta^*) = \mu_i(X_{t-}, \vartheta^*) + \sum_{j=1}^d \int_{[0,t) \times \bbX} \langle A_{ij}(\vartheta^*) \arrowvert e^{-sB_{ij}(\vartheta^*)}\rangle g_{ij}(x, \vartheta^*) \bar{N}^j(ds, dx),
\end{eqnarray}}
\hspace{-5pt}for $i=1, \dots, d$, where $\mu_i \colon \bbX \times \Xi \to \real_+$, $g_{ij} \colon \bbX \times \Xi \to \real_+, A_{ij} \colon  \Xi \to \real^{p \times p}, B_{ij} \colon \Xi \to \real^{p \times p}$ are measurable functions, and the real parts of the eigenvalues of $B_{ij}(\vartheta)$ are dominated by some $r > 0$ independently of $\vartheta \in \Xi$, for some $p \in \naturals$ and each $i,j = 1, \dots, d$.
We assume that there exist the Feller transition kernels of the mark process similar to (\ref{GEMHP eq 1}).
Moreover, we restrict the structure of the mark transition kernel $Q_j(x ,y, \vartheta)$.
We assume that there exists a dominating measure $\rho$ on $\bbX$ which induces the density $p_j$, i.e.,
\[
	Q_j(x ,dy, \vartheta) = p_j(x ,y, \vartheta) \rho(dy)
\]
for any $\vartheta \in \Xi$ and $j = 1, \dots, d$.
Then, we have the following representation of the compensator of $\bar{N}^i(ds, dx)$:
\[
	\nu^i(ds, dx) = \lambda^i_s(\vartheta^*) p_i(X_{s-} ,x, \vartheta^*) ds \rho(dx),
\]
where $X$ is the mark process defined as in Subsection \ref{sub Marked point process}.
We write $q^i_t(x, \vartheta) = p_i(X_{t-} ,x, \vartheta)$.
Let $T>0$ be an observation time index.
We consider the quasi log-likelihood function with respect to the above GEMHP:
{\small\begin{eqnarray}
	\label{App to GEMHP Eq 2}
	l_T(\vartheta) &=& \sum_{i=1}^d \left( \int_{[0, T] \times \bbX}   \log \big( \lambda^i_t(\vartheta)q^i_t(x, \vartheta) \big) \bar{N}^i(dt, dx)  -  \int_{[0, T] \times \bbX}  \lambda^i_t(\vartheta) q^i_t(x, \vartheta) dt \rho(dx) \right) \nonumber\\
	&=& \sum_{i=1}^d \left( \int_0^T \log \lambda^i_t(\vartheta) dN^i_t - \int_0^T \lambda^i_t(\vartheta) dt \right) + \sum_{i=1}^d \int_{[0, T] \times \bbX}  \log q^i_t(x, \vartheta) \bar{N}^i(dt, dx) \nonumber\\
	&=:& l_T^{(1)}(\vartheta) + l_T^{(2)}(\vartheta).
\end{eqnarray}}
\hspace{-5pt}Here, $l_T^{(1)}(\vartheta)$ is the part related to the counting process $N_t = \bar{N}([0,t] \times \bbX)$ and $l_T^{(2)}(\vartheta)$ to the mark process $X_t$.
We remark that $l_T(\vartheta)$ is actually equivalent to the log-likelihood function related to $\bar{N}$ if the filtration $\bbF$ satisfies the appropriate condition, see the condition (2.12) and Theorem 5.43 in \cite{JacodShiryaev2003}.
However, we call $l_T(\vartheta)$ the "quasi" log-likelihood function since we allow here to take a more general filtration $\bbF$.
We consider the P-O estimator whose first and third objective functions are set to this quasi log-likelihood function.
The quasi maximum likelihood estimator (QMLE) $\tilde{\vartheta}_T$ is defined as a quantity satisfying
\[
	\tilde{\vartheta}_T = \big( \tilde{\theta}^0_T, \tilde{\theta}^1_T, \tilde{\nu}^0_T, \tilde{\nu}^1_T \big)  \in \argmax_{\vartheta \in \Xi} l_T(\vartheta),
\]
and we set this QMLE as the first-step estimator.
Similar to as in Section \ref{P-O}, we set the second-step estimator of $(\theta^0, \nu^0)$ by
\[
	\big( \Stheta, \Snu \big) = \argmin_{ (\theta^0, \nu^0 ) \in \bar{\Theta}^0 \times \bar{\calN}^0} Q^{(q)}_T(\theta^0, \nu^0),
\]
and the third-step estimator of $\vartheta$ by
\[
	\check{\vartheta}_T = \big( \check{\theta}^0_T, \check{\theta}^1_T, \check{\nu}^0_T, \check{\nu}^1_T \big) = \argmax_{ \vartheta \in \hat{\Theta}^0_T \times \Theta^1 \times \hat{\calN}^0_T \times \calN^1} l_T(\vartheta),
\]
where $Q^{(q)}_T(\theta^0, \nu^0)$ is the same as (\ref{P-O Eq 2-1}), and $\hat{\Theta}^0_T, \hat{\calN}^0_T$ are defined in Step 3 of the P-O estimator in Section \ref{P-O}.
\begin{remark}
The results in this section also hold if we consider the quasi Bayesian estimator (QBE)
\[
	\tilde{\vartheta}_T = \frac{\int_{\Xi} \vartheta \exp(l_T(\vartheta)) p(\vartheta) d\vartheta}{\int_{\Xi} \exp(l_T(\vartheta)) p(\vartheta) d\vartheta}
\]
instead of the QMLE by changing notations in Step 1 and Step 3 in Section \ref{P-O}, where $p \colon \Xi \to \real_+$ is a continuous prior density with $0 < \inf_{\vartheta \in \Xi}p(\vartheta) < \sup_{\vartheta \in \Xi} p(\vartheta) < \infty$.
\end{remark}

We define the quasi likelihood ratio random field as
\begin{eqnarray}
\label{App to GEMHP Eq 3}
	\bbZ_T(u, \nu^*) = \exp\big\{ l_T(\theta^* + u/\sqrt{T}, \nu^*) - l_T(\vartheta^*)\big\}.
\end{eqnarray}
As mentioned in Remark \ref{P-O Rem 1}, if $\bbZ_T$ satisfies the polynomial type large deviation inequality (PLD), we can show that Assumption \ref{P-O Asm 1} (i)' and (v) holds for the above P-O estimator.
Therefore, to obtain the PLD, we prepare additional conditions [AH1]-[AH3] beside the conditions [L1]-[L2], [ND1]-[ND2].
We sometimes rewrite $\vartheta = (\theta^0, \theta^1, \nu^0, \nu^1)$ as $(\theta, \nu)$, where $\theta = (\theta^0, \theta^1) \in \Theta = \bar{\Theta}^0 \times \Theta^1$ and $\nu = (\nu^0, \nu^1) \in \calN = \bar{\calN}^0 \times \calN^1$.
We also write $\vartheta^*= (\theta^*, \nu^*) = (\theta^{0*}, \theta^{1*}, \nu^{0*}, \nu^{1*})$, where $(\theta^{0*}, \theta^{1*})$ is the true value and $(\nu^{0*}, \nu^{1*})$ is an arbitrary point in $\calN$.
We call that $f(\vartheta)$ is of class $C^i(\Xi)$ for some $i \in \naturals$ if $f(\vartheta)$ is of class $C^i(\mathring{\Xi})$ and its derivatives admit continuous extensions on $\partial \Xi$.
\begin{description}
	\item[\textnormal{[AH1]}]
		\begin{enumerate}
			\itemi  For any $x, y \in \bbX$ and $i, j = 1, \dots, d$, $\mu_i(x, \cdot)$, $g_{ij}(x, \cdot)$, $A_{ij}(\cdot)$, $B_{ij}(\cdot)$, $p_i(x, y, \cdot)$ are in $C^4(\Xi)$.
			\itemii For any $\vartheta \in \Xi$ and $i = 1, \dots, d$, $ \lambda^i_t(\vartheta)q^i_t(x, \vartheta) = 0$ if and only if $ \lambda^i_t(\vartheta^*)q^i_t(x, \vartheta^*) =0, dt\rho(dx)P(d\omega)\text{-}a.e.$
		\end{enumerate}
	\item[\textnormal{[AH2]}] For any $p>1$, $x \in \bbX$, and $i, j = 1, \dots, d$, there exists $C_p>0$ which may depend on $p$ such that
  \begin{enumerate}
			\itemi {\small $\sup_{\vartheta \in \Xi} \sum_{n=0}^3 \big[ \lvert \partial^n_{\theta} \mu_i(x, \vartheta) \rvert^p + \lvert \mu_i(x, \vartheta) 1_{\{ \mu_i(x, \vartheta) \neq 0 \}} \rvert^{-p} + \lvert \partial ^n_{\theta}g_{ij} (x, \vartheta) \rvert^p \big] \le C_p e^{\eta f_X(x)}$,}
			\itemii {\small $\sup_{\vartheta \in \Xi, \nu^* \in \calN} \sum_{n=0}^3 \int_{\bbX} \lvert \partial^n_{\theta} \log p_i (x, y, \vartheta) \rvert^p p_i (x, y, \vartheta^*) \rho(dy) \le C_p e^{\eta f_X(x)}$,}
		\end{enumerate}
		where $f_X$ is a norm-like function in [L1], and $\eta$ is a positive constant in Theorem \ref{GEMHP thm 1}.
\end{description}
The conditions [AH1] and [AH2] are sufficient conditions to leads the conditions [A1]-[A3] in Appendix \ref{Appendix for Sec4}.
[AH1] is the regularity condition for the intensity process, which leads to [A1].
[AH2] is closely related to [A2] and is used for the evaluation of moments of the quasi log-likelihood process.
Let the quasi log-likelihood random field be
\[
	\bbY_T(\vartheta, \nu^*) = \frac{1}{T} \big\{ l_T(\vartheta) - l_T(\vartheta^*) \big\}.
\]
The $V$-geometric ergodicity of the GEMHP and [AH1]-[AH2] guarantee the existence of the limit field
\[
	\bbY(\vartheta, \nu^*) = \sum_{i=1}^d E \left[ \int_{\bbX} \log \left( \frac{ f'^i_0(x, \vartheta) }{ f'^i_0(x, \vartheta^*) } \right) f'^i_0(x, \vartheta^*) - \left(f'^i_0(x, \vartheta) - f'^i_0(x, \vartheta^*)\right) \rho(dx) \right]
\]
where $f'^i_0(x, \vartheta) = \lambda'^i_0(\vartheta)q'^i_0(x, \vartheta)$ is the density of the predictable compensator of the stationary version $\bar{N}'$ at time $0$, see Lemma 4.1 in \cite{Clinet2021}.
For this $\bbY(\vartheta, \nu^*)$, we consider the following identifiability condition.
\begin{description}
	\item[\textnormal{[AH3]}] $\inf_{\theta \in \Theta-\{\theta^*\}, \nu, \nu^* \in \calN} -\frac{\bbY(\vartheta, \nu^*)}{\lvert\theta - \theta^*\rvert^2} > 0$.
\end{description}
Moreover, we write $\Gamma(\nu^*) = -\partial_{\theta}^2 \bbY(\vartheta^*, \nu^*)$.
\begin{remark}
\label{App to GEMHP Rem 1}
	$\bbY(\vartheta, \nu^*)$ is of class $C^2(\Xi)$ by the condition [A3] in Appendix.
	Then, the condition [AH3] is equivalent to the two following conditions:
	\begin{enumerate}
		\itemi $\bbY(\vartheta, \nu^*) < 0$ for any $\theta \in \Theta-\{\theta^*\}$ and $\nu, \nu^* \in \calN$.
		\itemii $\Gamma(\nu^*)$ is positive definite uniformly in $\nu^* \in \calN$.
	\end{enumerate}
\end{remark}
Let $U_T = \{u \in \real^{p_0+p_1} \arrowvert \theta^* + u/\sqrt{T} \in \Theta \}$ and $V_T(r) = \{u \in U_T \arrowvert \lvert u \rvert \ge r\}$.
The following theorem says that the quasi likelihood ratio random field $\bbZ_T$ satisfies the PLD.
\begin{theorem}[Polynomial type large deviation inequality]
\label{App to GEMHP Thm 1}
	Under [L1]-[L2], [ND1]-[ND2], and [AH1]-[AH3], for any $L>0$, there exists $C_L>0$ such that
	\[
		\sup_{r>0, T>0} P\left[ \sup_{u \in V_T(r), \nu^* \in \calN} \bbZ_T(u, \nu^*) \ge e^{-r} \right] \le \frac{C_L}{r^L},
	\]
	where $\bbZ_T(u, \nu^*)$ is defined in (\ref{App to GEMHP Eq 3}). In particular, $\sqrt{T} \big( \tilde{\theta}^0_T - \theta^{0*} \big)$ is $L^{\infty-}$-bounded.
\end{theorem}

On the other hand, the asymptotic normality of the QMLE, that is, Assumption \ref{P-O Asm 1} (iii), is shown in Corollary 4.2 of \cite{Clinet2021}.
With the help of the PLD and the asymptotic normality, we obtain the following theorem, which is the main theorem in this article.

\begin{theorem}
\label{App to GEMHP Thm 2}
	Under [L1]-[L2], [ND1]-[ND2], and [AH1]-[AH3], Assumption \ref{P-O Asm 1} are satisfied by $\bbL^1_T(\vartheta) = \bbL^2_T(\vartheta) = l_T(\vartheta)$ and $\alpha_T, \ep_T$ in Remark \ref{P-O Rem 2}.
	In particular, following statements hold:
	\begin{enumerate}
		\itemi $\sqrt{T}\big\{ ( \check{\theta}^0_{\calJ^1}, \check{\theta}^1_T ) - ( \phi^*, \theta^{1*} ) \big\} \to^d \Lambda^{-1/2}(\nu^*) \xi$ as $T \to \infty$, where $\Lambda(\nu^*) = -\partial_{\theta_{\calJ^1}}^2 \bbY(\vartheta^*, \nu^*)$ and $\xi$ is a  standard Gaussian random vector.
		\itemii $\sqrt{T} \big\{ (\check{\theta}^0_T, \check{\theta}^1_T ) - (\theta^{0*}, \theta^{1*}) \big\}$ is $L^{\infty-}$-bounded. Moreover, for all $L>0$, there exists $C_L >0$ such that
		\[
			P\left[ \check{\calJ}^0_T = \calJ^0 \right] \ge 1 - C_L T^{-L}
		\]
		for all $T>0$, where $\calJ^0 = \big\{ j = 1, \dots, p_0 \big\arrowvert \theta^{0*}_j = 0 \big\}$ and $\check{\calJ}^0_T = \big\{ j = 1, \dots, p_0 \big\arrowvert \check{\theta}^0_j = 0 \big\}$.
	\end{enumerate}
\end{theorem}


\section{Examples and Simulation Results}
\label{Ex and Sim}
In this section, we show the numerical simulation result for Theorem \ref{App to GEMHP Thm 2}.
Experiments in the scenario with no zero parameters and a comparison with previous studies are presented in Appendix \ref{Appendix Sim}.
All experiments are done by using Python3.\footnote{The code is available on the GitHub page \url{https://github.com/goda235/Sparse-estimation-for-GEMHP}.}
We consider the same notations and restrictions in Section \ref{App to GEMHP}.
Moreover, let $\Phi(x)$ be as in Section \ref{GEMHP}.
We write $Z_t = (\calE_t, X_t)$, where $X_t$ is the mark process and $\calE_t$ is the generalized elementary excitation process defined by (\ref{GEMHP Eq 1}).
The value of the GEMHP's intensity process $\lambda_t$ is updated based on the value of $Z_u$ for $u<t$ as follows:
{\small\begin{eqnarray*}
	\lambda^i_t
  = \mu_i(X_{t-}) + \sum_{j=1}^d \bigg\{\big\langle A_{ij} \big\arrowvert e^{-(t-u)B_{ij}} \mathcal{E}^{ij}_u \big\rangle + \int_{[u,t) \times \bbX} \langle A_{ij} \arrowvert e^{-sB_{ij}}\rangle g_{ij}(x) \bar{N}^j(ds, dx) \bigg\},
\end{eqnarray*}}
\hspace{-5pt}for $i = 1, \dots, d$.
Then, a path of the GEMHP can be simulated by using the above computation and Ogata's method, see \cite{Ogata1988}.
In the following subsections, we see the oracle properties of the P-O estimator for the GEMHP via numerical experiments for two basic models.
For each model, we calculate the QMLE and the P-O estimator 300 times, respectively, while changing the observation time to $T=100, 500$, and $3000$.
Here, all optimizations are done by the limited memory Broyden-Fletcher-Goldfarb-Shanno method for bound-constrained (L-BFGS-B), see \cite{ZhuEtal1997}.


\subsection{Multivariate Exponential Hawkes process}
\label{Ex and Sim 1}
First, we focus on the non-marked, however, quite important, Hawkes process.
\subsubsection{Definition}
\label{Ex and Sim 1-1}
We consider the multivariate exponential Hawkes process $N_t = (N^1_t, \dots, N^d_t)$ with the intensity
\begin{eqnarray}
	\label{Ex and Sim Eq 1-1-1}
	\lambda^i_t(\vartheta)\big\arrowvert_{\vartheta = \vartheta^*} = \mu_i + \sum_{j=1}^d \int_{[0,t)} \alpha_{ij}e^{-(t-s)\beta_{ij}} N^j(ds) \bigg\arrowvert_{\vartheta = \vartheta^*},
\end{eqnarray}
for $i = 1,\dots, d$, where $\vartheta = \big((\alpha_{ij})_{ij}, (\mu_i)_i, (\beta_{ij})_{ij} \big) \in \Xi$ is a parameter, $\vartheta^* = \big( (\alpha^*_{ij})_{ij}, (\mu^*_i)_i, (\beta^*_{ij})_{ij} \big) \in \Xi$ is the true value, and $\Xi = \Theta_{\alpha} \times \Theta_{\mu} \times \Theta_{\beta} \subset \real_+^{d^2} \times \real_{>0}^d \times \real_{>0}^{d^2}$ is an open convex bounded parameter space.
We assume the following conditions.
\begin{assumption}
\label{Ex and Sim Asm 1-1-1}
	\begin{enumerate}
		\itemi Some $\alpha^*_{ij}$ might be $0$ besides all $\mu^*_i$ and $\beta^*_{ij}$ are positive. Moreover, $\bar{\Theta}_{\beta} \subset \real^{d^2}_{>0}$.
		\itemii The spectral radius of $\Phi = \Big(\frac{\alpha^*_{ij}}{\beta^*_{ij}}\Big)_{ij}$ is less than $1$.
	\end{enumerate}
\end{assumption}
In (i), we assume that $\beta_{ij}$ are away from $0$ to control the oscillation of the nuisance parameter.
(ii) is a stability condition related to [L2].
We write $\theta^0 = (\alpha_{ij})_{ij}$, $\theta^1 = \big((\mu_i)_i, (\beta_{ij} \arrowvert i,j \ s.t. \ \alpha^*_{ij} \neq 0)\big)$, and $\nu^1 = (\beta_{ij} \arrowvert i,j \ s.t. \ \alpha^*_{ij} = 0)$.
The multivariate exponential Hawkes process has been used in various fields, for example, a model of limit order books in finance (see \cite{AbergelEtal2016}), a model of the time of posting texts on social media (see \cite{GodaEtal2021}), etc.
The directed graph whose each vertice corresponds to the value of $\mu_i^*$ and each edge corresponds to the excitability value of $\alpha_{ij}^*/\beta_{ij}^*$ is called the Hawkes graph, see \cite{EmbrechtsKirchner2018}.
The Hawkes graph allows visualizing the structure of the mutual excitations in a model.
The $\alpha^*_{ij}=0$ means that the effect of an event $j$ on the probability of the occurrence of an event $i$ is zero, i.e., the edge from $j$ to $i$ is disconnected in a Hawkes graph, see Figure \ref{Ex and Sim Fig1-1-1}.
\begin{figure}[htbp]
	\center
	\includegraphics[width=5cm]{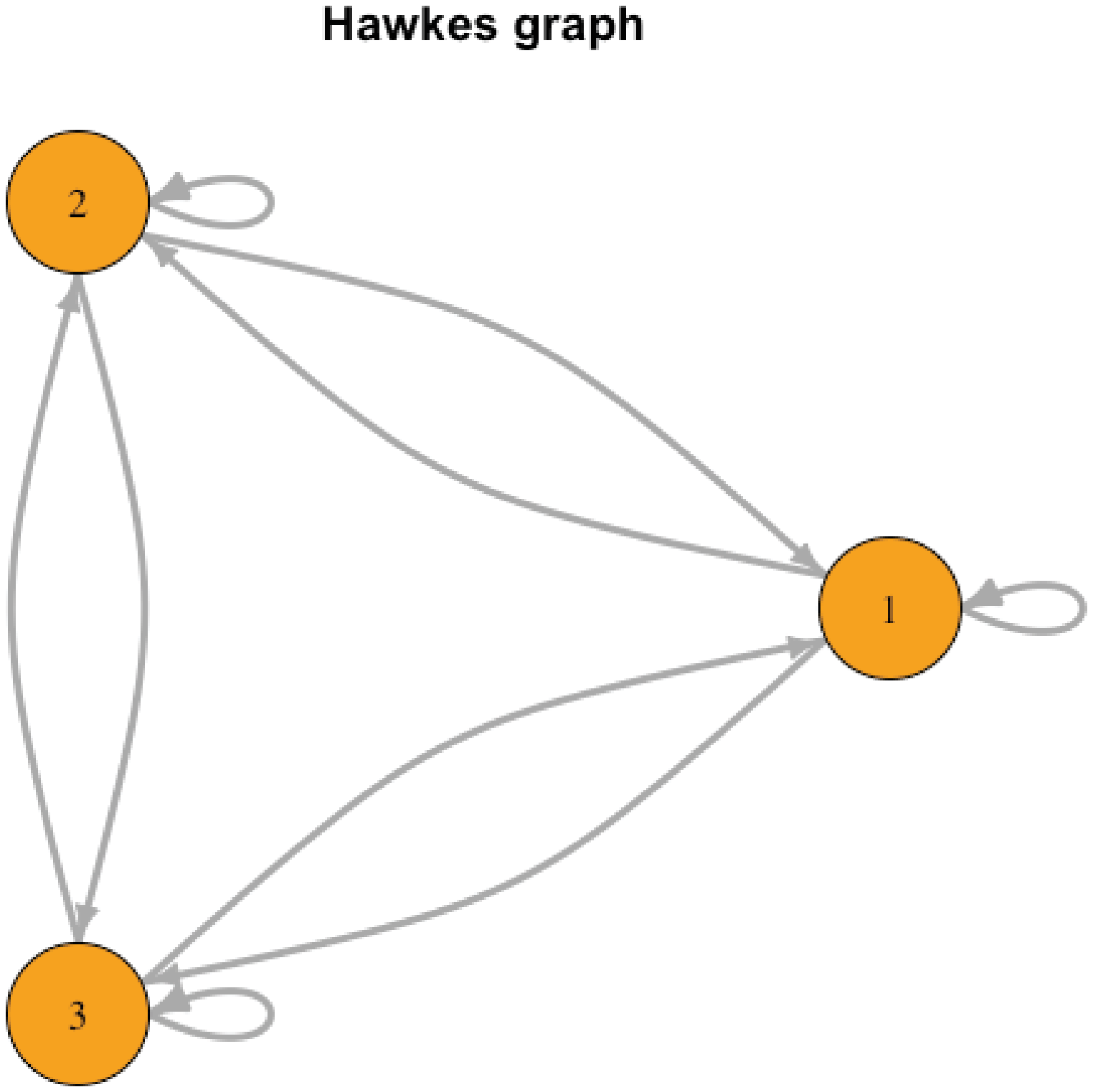}
	\includegraphics[width=5cm]{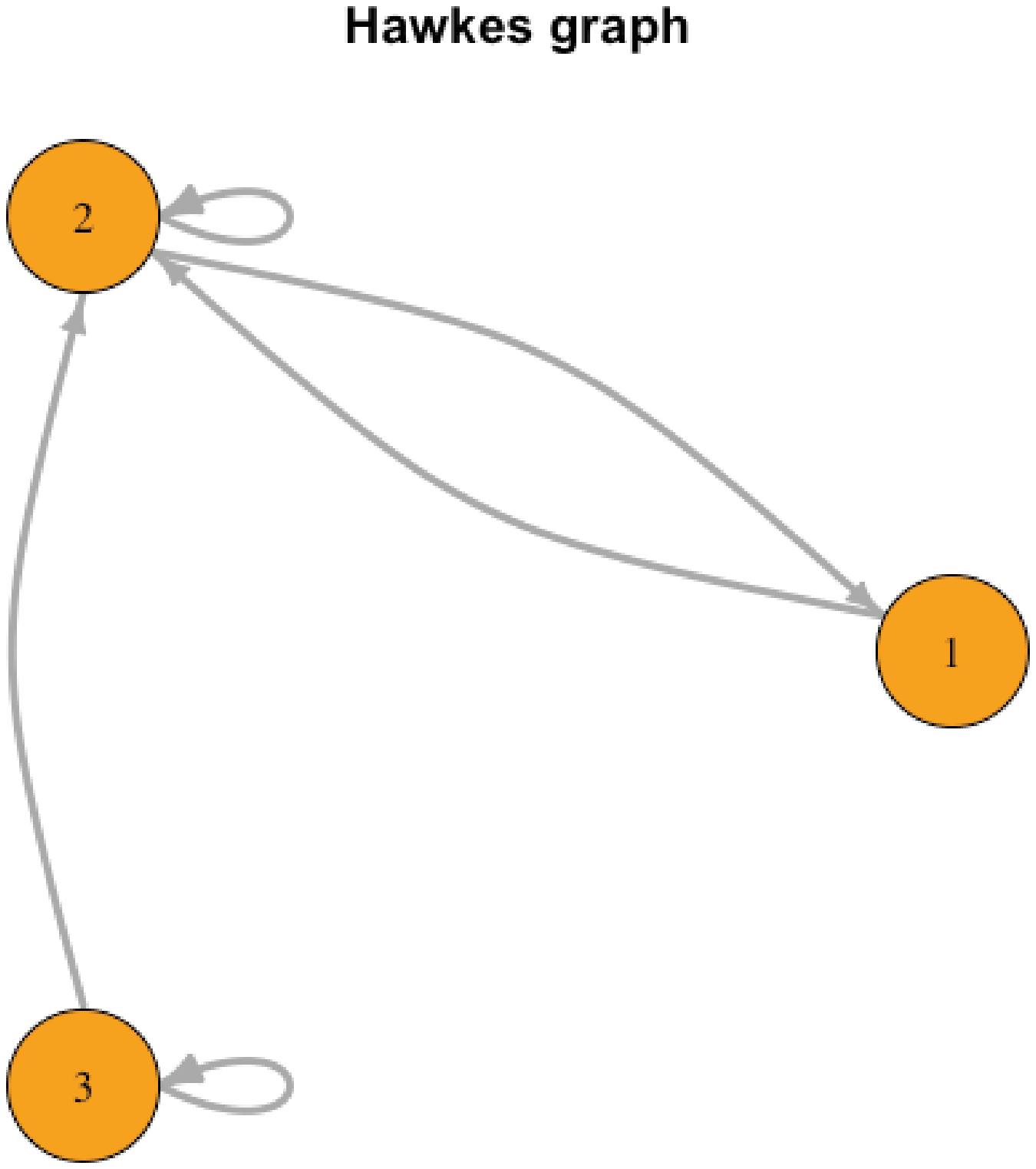}
	\caption{{\bf Examples of Hawkes graphs by the 3-dimensional multivariate exponential Hawkes process.} The left figure shows the case where all $\alpha^*_{ij}$'s are positive. The right figure shows the case where $\alpha^*_{11}=\alpha^*_{13}=\alpha^*_{31}=\alpha^*_{32}=0$.}
	\label{Ex and Sim Fig1-1-1}
\end{figure}
For the multivariate exponential Hawkes process, the oracle properties of the P-O estimation holds by Theorem \ref{App to GEMHP Thm 2}.
\begin{proposition}
\label{Ex and Sim Pro 1-1-1}
	Under Assumption \ref{Ex and Sim Asm 1-1-1}, the multivariate exponential Hawkes process with the intensity (\ref{Ex and Sim Eq 1-1-1}) satisfies the conditions \textnormal{[L1]-[L2], [ND1]-[ND2], and [AH1]-[AH3]}.
\end{proposition}

\subsubsection{Simulation Results}
\label{Ex and Sim 1-2}

Let $N_t = (N^1_t ,N^2_t, N^3_t)$ be a $3$-dimensional exponential Hawkes process with the following parameters:
\[
	\mu^* = (0.2, 0.1, 0.1), \quad
	\alpha^* = \left(
		\begin{array}{ccc}
			0.0& 0.2& 0.0\\
			0.2& 0.1& 0.4\\
			0.0& 0.0& 0.2
		\end{array}
		\right), \quad
	\beta^* = \left(
		\begin{array}{ccc}
			*& 0.9& *\\
	 		0.5& 1.2& 0.6\\
	 		*& *& 0.7
		\end{array}
		\right),
\]
where $*$ means a non-definite value.
We set the hyperparameters of the P-O estimator in Remark \ref{P-O Rem 2} to be $q=1.0, \gamma=1.0, a=0.5$, the observation times $T = 100, 500, 3000$, and the number of the Monte Carlo simulation $MC = 300$.
Table \ref{Ex and Sim Table1-2-1} shows the fraction of trials in which the parameter $\alpha_{ij}$'s are estimated to be completely zero.
We can see that the variable selection is performed more accurately by the P-O estimator than by the QMLE as $T$ becomes larger.

\begin{table}[htbp]
  \centering
	\caption{{\bf Percentage of estimated to be zero.}}
	\label{Ex and Sim Table1-2-1}
  \scalebox{0.85}{
	\begin{tabular}{| l | l | l | l | l | l |} \hline
		\multicolumn{6}{| l |}{QMLE: T=100 } \\ \hline
		$\alpha_{11}$ & {\bf 71.7\%} & $\alpha_{12}$ & 11.7\% & $\alpha_{13}$ &  {\bf 52.7\%} \\ \hline
		$\alpha_{21}$ & 5.67\% & $\alpha_{22}$ & 51.3\% & $\alpha_{23}$ & 5.00\% \\ \hline
		$\alpha_{31}$ & {\bf 60.0\%} & $\alpha_{32}$ & {\bf 58.3\%} & $\alpha_{33}$ & 26.7\% \\ \hline
		\multicolumn{6}{| l |}{QMLE: T=500} \\ \hline
		$\alpha_{11}$ & {\bf 62.7\%} & $\alpha_{12}$ & 0.00\% & $\alpha_{13}$ &  {\bf 54.0\%} \\ \hline
		$\alpha_{21}$ & 0.00\% & $\alpha_{22}$ & 11.3\% & $\alpha_{23}$ & 0.00\% \\ \hline
		$\alpha_{31}$ & {\bf 58.7\%} & $\alpha_{32}$ & {\bf 54.7\%} & $\alpha_{33}$ & 0.00\% \\ \hline
		\multicolumn{6}{| l |}{QMLE: T=3000} \\ \hline
		$\alpha_{11}$ & {\bf 58.0\%} & $\alpha_{12}$ & 0.00\% & $\alpha_{13}$ &  {\bf 50.0\%} \\ \hline
		$\alpha_{21}$ & 0.00\% & $\alpha_{22}$ & 0.00\% & $\alpha_{23}$ & 0.00\% \\ \hline
		$\alpha_{31}$ & {\bf 54.3\%} & $\alpha_{32}$ & {\bf 57.0\%} & $\alpha_{33}$ & 0.00\% \\ \hline
	\end{tabular}
	\begin{tabular}{| l | l | l | l | l | l |} \hline
		\multicolumn{6}{| l |}{P-OE: T=100 } \\ \hline
		$\alpha_{11}$ & {\bf 84.3\%} & $\alpha_{12}$ & 30.7\% & $\alpha_{13}$ &  {\bf 66.3\%} \\ \hline
		$\alpha_{21}$ & 27.3\% & $\alpha_{22}$ & 67.3\% & $\alpha_{23}$ & 11.0\% \\ \hline
		$\alpha_{31}$ & {\bf 80.3\%} & $\alpha_{32}$ & {\bf 82.0\%} & $\alpha_{33}$ & 42.0\% \\ \hline
		\multicolumn{6}{| l |}{P-OE: T=500 } \\ \hline
		$\alpha_{11}$ & {\bf 83.3\%} & $\alpha_{12}$ & 3.67\% & $\alpha_{13}$ &  {\bf 72.7\%} \\ \hline
		$\alpha_{21}$ & 1.67\% & $\alpha_{22}$ & 34.0\% & $\alpha_{23}$ & 0.00\% \\ \hline
		$\alpha_{31}$ & {\bf 84.7\%} & $\alpha_{32}$ & {\bf 83.7\%} & $\alpha_{33}$ & 6.33\% \\ \hline
		\multicolumn{6}{| l |}{P-OE: T=3000} \\ \hline
		$\alpha_{11}$ & {\bf 87.7\%} & $\alpha_{12}$ & 0.00\% & $\alpha_{13}$ &  {\bf 79.0\%} \\ \hline
		$\alpha_{21}$ & 0.00\% & $\alpha_{22}$ & 3.33\% & $\alpha_{23}$ & 0.00\% \\ \hline
		$\alpha_{31}$ &  {\bf 87.3\%} & $\alpha_{32}$ & {\bf 88.7\%} & $\alpha_{33}$ & 0.00\% \\ \hline
	\end{tabular}
  }
\end{table}

\begin{remark}
\label{Ex and Sim Rem 1-2-1}
	In Table \ref{Ex and Sim Table1-2-1}, the QMLE asymptotically correctly estimates about 50\% of the zero parameters.
	This phenomenon is derived from the local asymptotic normality on the restricted parameter space to a positive region, see \cite{GodaEtal2021} for an intuitive but more detailed explanation.
	There are few studies of the maximum likelihood method on the constrained parameter space using as a variable selection method.
	In an i.i.d. case, the asymptotic behavior of the maximum likelihood estimator is discussed under general assumptions where we can consider such a constrained parameter space, see \cite{LeCam1970}.
	The same phenomenon is observed in Table \ref{Ex and Sim Table2-2-1} below.
\end{remark}

Figure \ref{Ex and Sim Fig1-2-1} shows histograms of the error distribution of the P-O estimator, that is, histograms of the values of $\sqrt{T} (\check{\vartheta}_T - \vartheta^*)$.
It seems that the distribution is close to the normal distribution as $T$ becomes larger.
\begin{figure}[htbp]
	\centering
	\includegraphics[width=11.5cm]{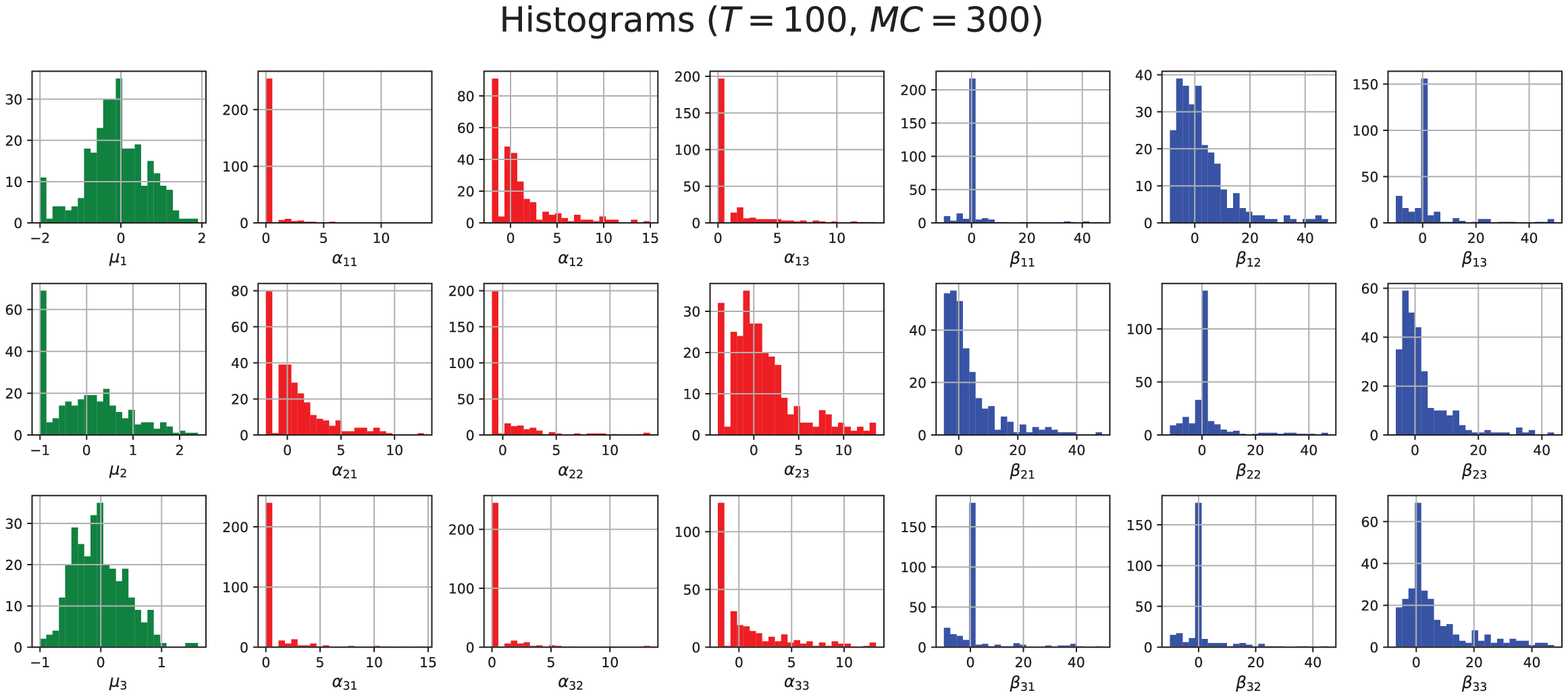}\\
  \vspace{-5mm}
	\includegraphics[width=11.5cm]{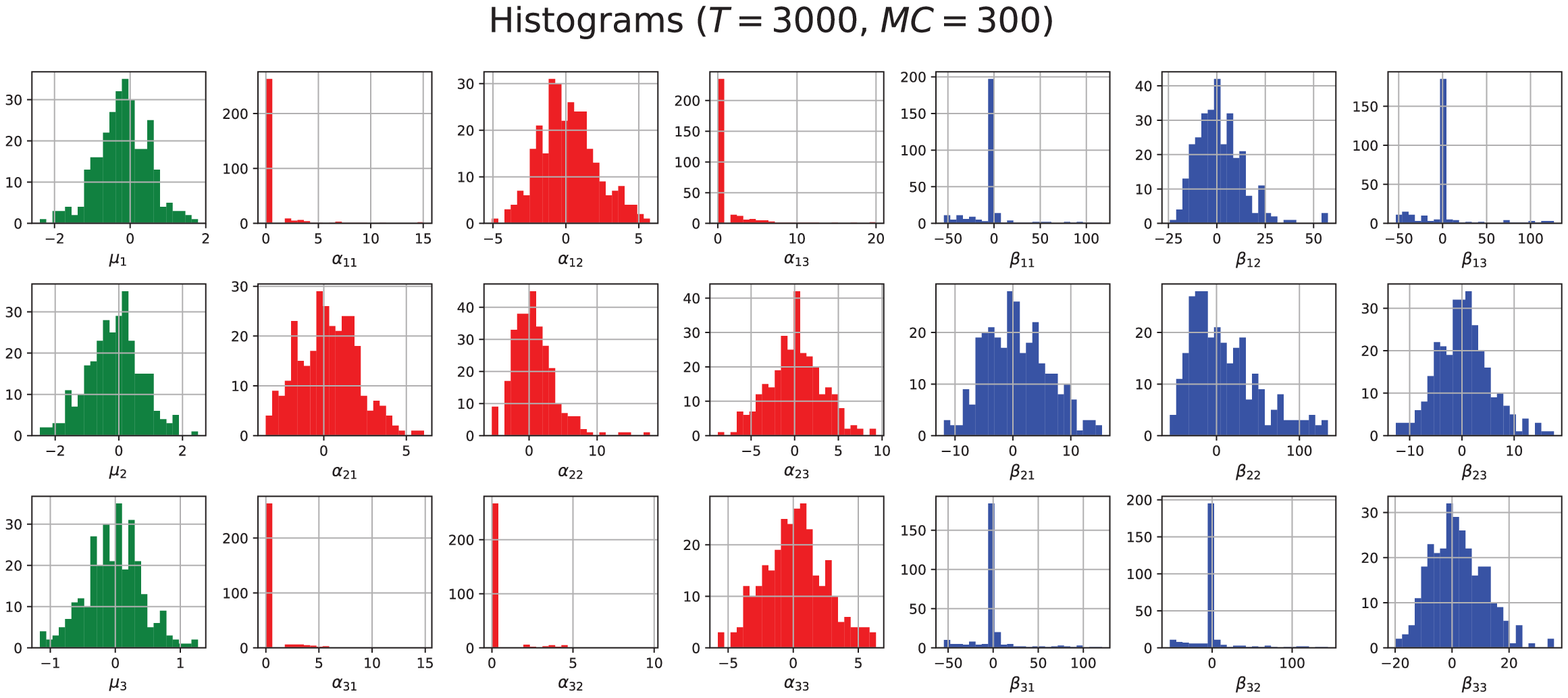}
	\caption{{\bf Histograms of the error of the P-O estimator $\sqrt{T}(\check{\vartheta}_T - \vartheta^*)$.}}
	\label{Ex and Sim Fig1-2-1}
\end{figure}

Table \ref{Ex and Sim Table1-2-2} shows the averages of squared errors of the QMLE $(\tilde{\vartheta}_T - \vartheta^*)^2$ and the P-O estimator $(\check{\vartheta}_T - \vartheta^*)^2$.
For non-zero parameters, owing to the asymptotic normality, both the QMLE and the P-O estimator have asymptotically the same level of variance.
For zero parameters, we can guess that the P-O estimator asymptotically has a smaller error than the QMLE, due to the accurate model selection.
However, when the observation time is small, the performance of the QMLE is better due to the miss model selection of the P-O estimator.

\begin{table}[htbp]
	{\footnotesize
 	\caption{{\bf Average of squared errors.}}
 	\label{Ex and Sim Table1-2-2}
  \begin{tabular}{ll | lll lll}
		\hline
   	$T$ & Method & $\mu_1$ & $\mu_2$ & $\mu_3$ & $\alpha_{11}$ & $\alpha_{12}$ & $\alpha_{13}$ \\ \hline\hline
   	100 & QMLE & 6.80e-03 & 5.63e-03 & 3.09e-03 & {\bf 5.90e-01} & 1.15e-00 & {\bf 4.89e-00} \\
   	& P-OE & 5.79e-03 & 6.73e-03 & 1.80e-03 & {\bf 6.34e-01} & 1.19e-00 & {\bf 5.01e-00} \\
   	500 & QMLE & 1.58e-03 & 1.81e-03 & 8.06e-04 & {\bf 1.61e-02} & 1.41e-02 & {\bf 1.03e-01} \\
   	& P-OE & 1.24e-03 & 1.80e-03 & 3.62e-04 & {\bf 2.16e-02} & 1.39e-02 & {\bf 9.51e-02} \\
   	3000 & QMLE & 2.59e-04 & 2.44e-04 & 1.84e-04 & {\bf 1.81e-03} & 1.28e-03 & {\bf 4.19e-03} \\
 	 	& P-OE & 1.68e-04 & 2.43e-04 & 6.17e-05 & {\bf 1.72e-03} & 1.26e-03 & {\bf 4.76e-03} \\
   	\hline
 	\end{tabular}
 	\vspace{2mm}

 	\begin{tabular}{ll | lll lll}
	 	\hline
   	$T$ & Method & $\alpha_{21}$ & $\alpha_{22}$ & $\alpha_{23}$ & $\alpha_{31}$ & $\alpha_{32}$ & $\alpha_{33}$ \\ \hline\hline
   	100 & QMLE & 3.78e-01 & 1.38e-00 & 1.47e-00 & {\bf 1.66e-00} & {\bf 5.75e-01} & 4.73e-01 \\
   	& P-OE & 4.89e-01 & 1.51e-00 & 1.77e-00 & {\bf 1.77e-00} & {\bf 1.06e-00} & 6.64e-01 \\
   	500 & QMLE & 9.07e-03 & 1.01e-01 & 2.26e-02 & {\bf 8.13e-03} & {\bf 5.51e-02} & 1.87e-02 \\
   	& P-OE & 9.26e-03 & 1.02e-01 & 2.20e-02 & {\bf 8.17e-03} & {\bf 5.50e-02} & 1.86e-02 \\
   	3000 & QMLE & 1.09e-03 & 3.81e-03 & 3.15e-03 & {\bf 1.07e-03} & {\bf 1.82e-03} & 1.95e-03 \\
 	 	& P-OE & 1.09e-03 & 3.96e-03 & 3.13e-03 & {\bf 1.03e-03} & {\bf 1.77e-03} & 1.91e-03 \\
   	\hline
 	\end{tabular}
 	\vspace{2mm}

 	\begin{tabular}{ll | lll lll}
	 	\hline
   	$T$ & Method & $\beta_{11}$ & $\beta_{12}$ & $\beta_{13}$ & $\beta_{21}$ & $\beta_{22}$ & $\beta_{23}$ \\ \hline\hline
   	100 & QMLE & * & 4.07e+02 & * & 6.88e+01& 2.77e+02 & 5.97e+01 \\
   	& P-OE & * & 4.17e+02 & * & 7.79e+01& 2.81e+02 & 8.79e+01 \\
   	500 & QMLE & * & 1.70e-00 & * & 1.84e-01 & 1.37e+01& 1.30e-01 \\
   	& P-OE & * & 1.51e-00 & * & 1.67e-01 & 1.36e+01& 1.38e-01 \\
   	3000 & QMLE & * & 5.67e-02 & * & 1.04e-02 & 6.99e-00 & 8.98e-03 \\
 	 	& P-OE & * & 5.13e-02 & * & 1.01e-02 & 6.94e-00 & 8.88e-03 \\
   	\hline
 	\end{tabular}
	\vspace{2mm}

 	\begin{tabular}{ll | lll}
	 	\hline
   	$T$ & Method & $\beta_{31}$ & $\beta_{32}$ & $\beta_{33}$ \\ \hline\hline
   	100 & QMLE & * & * & 4.66e+01 \\
   	& P-OE & * & * & 6.24e+01 \\
   	500 & QMLE & * & * & 1.76e-00 \\
   	& P-OE & * & * & 1.65e-00 \\
   	3000 & QMLE & * & * & 3.58e-02 \\
 	 	& P-OE & * & * & 3.14e-02 \\
   	\hline
 	\end{tabular}}
\end{table}

\subsection{Hawkes process marked with "Topic"}
\label{Ex and Sim 2}
Second, we introduce the marked Hawkes process useful in the field of natural language processing.
\subsubsection{Definition}
\label{Ex and Sim 2-1}
We consider a web service where $d$ types of users post texts while reading each other's posts, like social network services such as Twitter and Facebook, product reviews on Amazon, and so on.
In this subsection, we model a sequence of posting times $(T^i_n)_{i= 1, \dots, d, n \in \naturals}$ by using a GEMHP.
Since we can consider the distribution of future posting times naturally depends on the content of the previous posts, we will regard the content of texts as marks.

Techniques to quantify the amount of "topic" in a sentence have been studied in the field of natural language processing.
For example, Latent Dirichlet Allocation (LDA) is a hierarchical Bayesian model in which a sentence $w = (w_1, \dots, w_N)$ consisting of $N$ words is generated by a conditional multinomial distribution given a "topic" $z$, see \cite{DavidEtal2003}.
The topic $z$ in the LDA model is a $\{1, \dots, d'\}$-valued random variable (where $d'$ is the number of topics), and its distribution is a conditional multinomial distribution whose parameters are generated by the Dirichlet distribution.
Conversely, we can consider the conditional probabilities $\big( p(z=l \arrowvert w) \big)_{l=1,\dots,d'}$, and it can be assumed as a proportion of each topic in a sentence $w$.

For $i = 1, \dots, d$ and $n \in \naturals$, let $w^i_n = (w^i_{n,1}, \dots, w^i_{n,M_n})$ be the $n$-th post by the $i$-th user, where $M_n$ is the number of words in a post $w^i_n$.
We assume that the conditional probabilities $\big(p(z=l \arrowvert w=w^i_n)\big)_{l,i,n}$ are given for a $\{1, \dots, d'\}$-valued random topic $z$.
Then, we regard $(p(z=1 \arrowvert w=w^i_n), \dots, p(z=d' \arrowvert w=w^i_n))$ as a mark $X^i_n$.

Now, we consider the model for the above sequence of a couple $(T^i_n, X^i_n)_{i,n}$.
Let $\bar{N}$ be the GEMHP with the intensity
\begin{eqnarray}
\label{Ex and Sim Eq 2-1-1}
	\lambda^i_t(\vartheta^*) = \mu_i + \sum_{j=1}^d \int_{[0,t) \times \mathbb{X}} e^{-(t-s)\beta_{ij}}\Bigg( \sum_{l=1}^{d'} m_{ijl}x_{l} \Bigg) \bar{N}^j(ds, dx) \bigg\arrowvert_{\vartheta = \vartheta^*},
\end{eqnarray}
for $i = 1, \dots, d$, and suppose that its mark process takes values on the $(d'-1)$-simplex\footnote{The $(d'-1)$-simplex is the set $\big\{ x= (x^1, \dots, x^{d'}) \in \real_+^{d'} \big\arrowvert \sum_{i=1}^{d'}x^i = 1 \big\}$.} and has the transition kernel
\[
	Q_j(x ,dy, \theta_M) = p_j(x ,y, \theta_M) dy,
\]
where $\vartheta = \big( (m_{ijl})_{ijl}, (\mu_i)_i, (\beta_{ij})_{ij}, \theta_M \big) \in \Xi$
is a parameter, $\vartheta^*= \big((m^*_{ijl})_{ijl},$ $(\mu^*_i)_i, (\beta^*_{ij})_{ij}, \theta_M^*\big) \in \Xi$ is the true value, and $\Xi = \Theta_m \times \Theta_{\mu} \times \Theta_{\beta} \times \Theta_M \subset \real_+^{d^2 \times d'} \times \real_{>0}^{d} \times \real_{>0}^{d^2} \times \real^{d''}$ is an open convex bounded parameter space.
For a vector $x = (x^1, \dots ,x^k) \in \real^k$, we write $x^{\otimes 2}$ as a tensor $(x^ix^j)_{i,j = 1, \dots, k} \in \real^{k \times k}$.
Recall that we write $q^i_t(x, \theta_M) = p_i(X_{t-} ,x, \theta_M)$.
When we have the geometric ergodicity of the above model, we can consider the stationary version of $q^i_t(x, \theta_M)$ and write it as $q'^{i}_t(x, \theta_M)$.
The following assumptions are sufficient conditions for [L1]-[L2], [ND1]-[ND2], and [AH1]-[AH3].
\begin{assumption}
\label{Ex and Sim Asm 2-1-1}
	\begin{enumerate}
		\itemi Some $m^*_{ijl}$ might be $0$ besides all $\mu^*_i$ and $\beta^*_{ij}$ are positive. Moreover, $\bar{\Theta}_{\beta} \subset \real^{d^2}_{>0}$.
		\itemii The spectral radius of $\Phi(x) = \Big(\frac{G_{ij}(x)}{\beta_{ij}}\Big)_{ij}$ is less than $1$ uniformly in $x \in \bbX$, where
		\[
			G_{ij}(x) = \int_{\bbX} \Bigg( \sum_{l=1}^{d'} m^*_{ijl}y_{l} \Bigg) p_j(x, y, \theta_M^*) dy.
		\]
		\itemiii The transition kernel $Q$ admits a reachable point $x_0 \in \bbX$. Moreover, there exists a lower semi-continuous function $r_j \colon \bbX^2 \to \real_+$ such that $Q_j$ admits a sub-component $\calT_j$ with $\calT_j(x, \bbX) > 0$ and $\calT_j(x, F) = \int_{F} r_j(x,y)dy$ for any $x \in \bbX$ and $F \in \calB(\bbX)$.
		\itemiv For any $l = 1, \dots, d'$, the mark process $X^l$ is not almost surely constant. Moreover, for any $l_1, l_2 = 1, \dots, d'$, $X^{l_1}$ and $X^{l_2}$ are distinguishable\footnote{That is, $P\big[X_t^{l_1} = X_t^{l_2} \text{ for all $t>0$}\big] < 1$.}.
		\itemv For any $p>1$, $x \in \bbX$, and $i= 1, \dots, d$,
    \[
      \sup_{\theta_M \in \Theta_M} \sum_{n=0}^3 \int_{\bbX} \lvert \partial^n_{\theta_M} \log p_i (x, y, \theta_M) \rvert^p p_i (x, y, \theta_M^*) \rho(dy) \le C_p e^{\eta f_X(x)}
    \]
    holds, where $f_X$ is a norm-like function in [L1], and $\eta$ is a positive constant in Theorem \ref{GEMHP thm 1}.
		Moreover, for any $t \ge 0$ and $i = 1, \dots, d$, the transition densities $q'^{i}_t(x, \theta_M)$'s satisfy the following conditions.
			\begin{enumerate}
        \renewcommand{\labelenumiii}{\alph{enumiii}).}
				\itemi If there exists $\theta_M \in \Theta_M$ such that $q'^{i}_t(x, \theta_M) = q'^{i}_t(x, \theta_M^*), \ dxP(d\omega)\text{-}a.e.$, then $\theta_M = \theta_M^*$.
				\itemii If there exists $\bbx \in \real^{d''}$ such that $\bbx^T\partial_{\theta_M}\log q'^{i}_t(x, \theta_M^*) = 0$, $dxP(d\omega)\text{-}a.e.$, then $\bbx = 0$.
				\itemiii Almost surely,
        {\small\begin{eqnarray*}
          \int_{\bbX} \big( \partial_{\theta_M} \log q'^{i}_t(x, \theta_M^*) \big)^{\otimes 2} q'^{i}_t(x, \theta_M^*)dx = -\int_{\bbX} \partial_{\theta_M}^2 \log q'^{i}_t(x, \theta_M^*)q'^{i}_t(x, \theta_M^*)dx.
        \end{eqnarray*}}
			\end{enumerate}
	\end{enumerate}
\end{assumption}
\begin{remark}
\label{Ex and Sim Rem 2-1-2}
	We can prove that the above GEMHP satisfies the conditions \textnormal{[L1]-[L2], [ND1]-[ND2]} without Assumption \ref{Ex and Sim Asm 2-1-1} (v).
	In other words, Theorem \ref{GEMHP thm 1} holds without Assumption \ref{Ex and Sim Asm 2-1-1} (v).
	Then, the statement of Assumption \ref{Ex and Sim Asm 2-1-1} (v) has a meaning under the other assumptions.
	See the proof of Proposition \ref{Ex and Sim Pro 2-1-1}.
\end{remark}
(i) is a constraint on the parameters, in particular, requiring $\beta_{ij}$ to be away from $0$ to control the oscillation of the nuisance parameter.
(ii) and (iii) are assumptions for the sake of [L2] and [ND2], respectively.
(iv) ensures the identifiability of the parameters.
(v) is an assumption related to the probability density of marks to guarantee [AH2] and [AH3].
For this model, the oracle properties of the P-O estimation holds by Theorem \ref{App to GEMHP Thm 2}.
\begin{proposition}
\label{Ex and Sim Pro 2-1-1}
	Under Assumption \ref{Ex and Sim Asm 2-1-1}, the GEMHP with the intensity (\ref{Ex and Sim Eq 2-1-1}) satisfies the conditions \textnormal{[L1]-[L2], [ND1]-[ND2], and [AH1]-[AH3]}.
\end{proposition}
\begin{remark}
\label{Ex and Sim Rem 2-1-1}
	Referring to the expression of the quasi log-likelihood function in (\ref{App to GEMHP Eq 2}), for the above model, the intensity's parameters $(m_{ijl})_{ijl}, (\mu_i)_i, (\beta_{ij})_{ij}$ and the mark's parameter $\theta_M$ are only included in $l^{(1)}_T(\vartheta)$ and $l^{(2)}_T(\vartheta)$, respectively.
	Therefore, we can obtain the estimator by optimizing $l^{(1)}_T(\vartheta)$ and $l^{(2)}_T(\vartheta)$ independently.
	Then, Assumption \ref{Ex and Sim Asm 2-1-1} (v) is irrelevant to the estimation for $(m_{ijl})_{ijl}, (\mu_i)_i, (\beta_{ij})_{ij}$.
\end{remark}

\subsubsection{Simulation Results}
\label{Ex and Sim 2-2}
As a simple case, suppose that only one user posts texts and set the number of topics in texts to $3$.
Moreover, we assume that the proportion of each topic in a text follows simply the Dirichlet distribution, although it has a more complicated distribution in the LDA model.

Let $\bar{N}$ be the $1$-dimensional GEMHP whose intensity is
\begin{eqnarray}
 	\lambda_t(\vartheta^*) &=& \mu + \int_{[0,t) \times \mathbb{X}} e^{-\beta(t-s)} \big(m_1x_1 + m_2x_2 + m_3x_3 \big)\bar{N}(ds, dx) \bigg\arrowvert_{\vartheta = \vartheta^*}\nonumber\\
 	&=& 1.5 + \int_{[0,t) \times \mathbb{X}} e^{-0.5(t-s)} \big(0.4x_1 + 0.0x_2 + 0.4x_3 \big)\bar{N}(ds, dx),\nonumber
\end{eqnarray}
where its marks independently and identically follow the $3$-dimensional Dirichlet distribution with a parameter $\alpha = (2, 2, 5)$.
Here, we only estimate the parameters $m_1, m_2, m_3, \mu, \beta$ since $\alpha$ is estimated as the conventional MLE, see Remark \ref{Ex and Sim Rem 2-1-1}.
Furthermore, we assume that only parameters $m$'s can take the zero value, i.e., we set $\theta^0 = (m_1, m_2, m_3)$ and $\theta^1 = (\mu, \beta)$.
Then, we can immediately confirm the conditions in Assumption \ref{Ex and Sim Asm 2-1-1}.
We set the hyperparameters of the P-O estimator in Remark \ref{P-O Rem 2} to be $q=1.0, \gamma=2.0, a=0.5$, the observation times $T = 100, 500, 3000$, and the number of the Monte Carlo simulation $MC = 300$.

Table \ref{Ex and Sim Table2-2-1} shows the fraction of trials in which the parameters $m_1, m_2$, and $m_3$ are estimated to be completely zero.
We can see that the variable selection is performed more accurately by the P-O estimator than by the QMLE as $T$ becomes larger.

\begin{table}[htbp]
	\centering
	\caption{{\bf Percentage of estimated to be zero.}}
	\label{Ex and Sim Table2-2-1}
  \scalebox{0.9}{
	\begin{tabular}{| l | l | l | l | l | l |} \hline
		\multicolumn{6}{| l |}{QMLE: $T=100$} \\ \hline
		$m_1$ & 30.0\% & $m_2$ & {\bf 58.7\%} & $m_3$ & 18.0\% \\ \hline
		\multicolumn{6}{| l |}{QMLE: $T=500$} \\ \hline
		$m_1$ & 6.00\% & $m_2$ & {\bf 54.0\%} & $m_3$ & 0.00\% \\ \hline
		\multicolumn{6}{| l |}{QMLE: $T=3000$} \\ \hline
		$m_1$ & 0.00\% & $m_2$ & {\bf 51.7\%} & $m_3$ & 0.00\%  \\ \hline
	\end{tabular}
	\begin{tabular}{| l | l | l | l | l | l |} \hline
		\multicolumn{6}{| l |}{P-OE: $T=100$} \\ \hline
		$m_1$ & 45.3\% & $m_2$ & {\bf 73.0\%} & $m_3$ & 50.3\% \\ \hline
		\multicolumn{6}{| l |}{P-OE: $T=500$} \\ \hline
		$m_1$ & 19.7\% & $m_2$ & {\bf 77.3\%} & $m_3$ & 7.33\% \\ \hline
		\multicolumn{6}{| l |}{P-OE: $T=3000$} \\ \hline
		$m_1$ & 0.00\% & $m_2$ & {\bf 85.7\%} & $m_3$ & 0.00\%  \\ \hline
	\end{tabular}
  }
\end{table}

Figure \ref{Ex and Sim Fig2-2-1} show histograms of the error distribution of the P-O estimator, that is, histograms of the values of $\sqrt{T} (\check{\vartheta}_T - \vartheta^*)$.
It seems that the distribution is close to the normal distribution as $T$ becomes larger.
\begin{figure}[htbp]
	\centering
	\includegraphics[width=11.5cm]{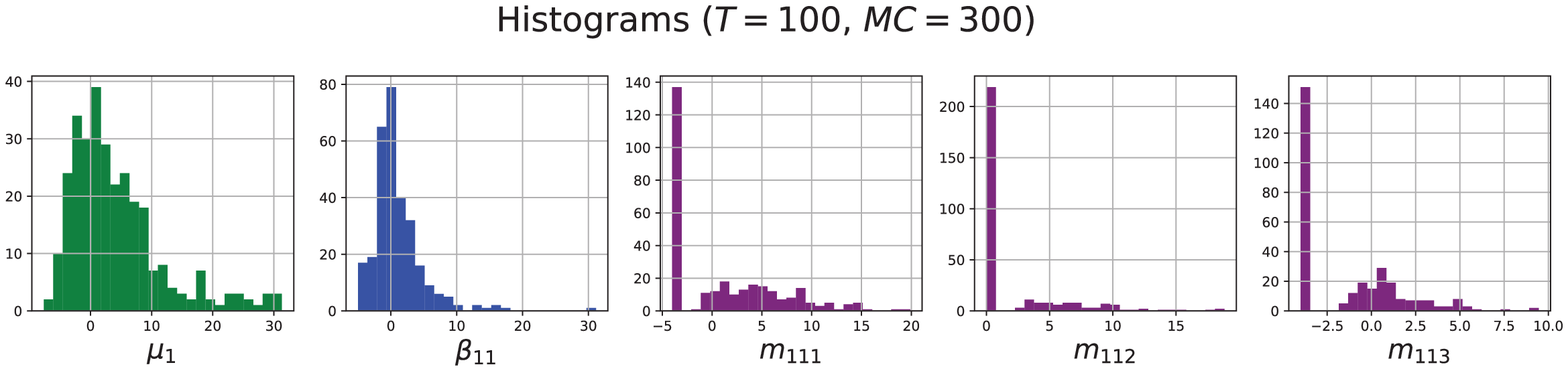}\\
	\includegraphics[width=11.5cm]{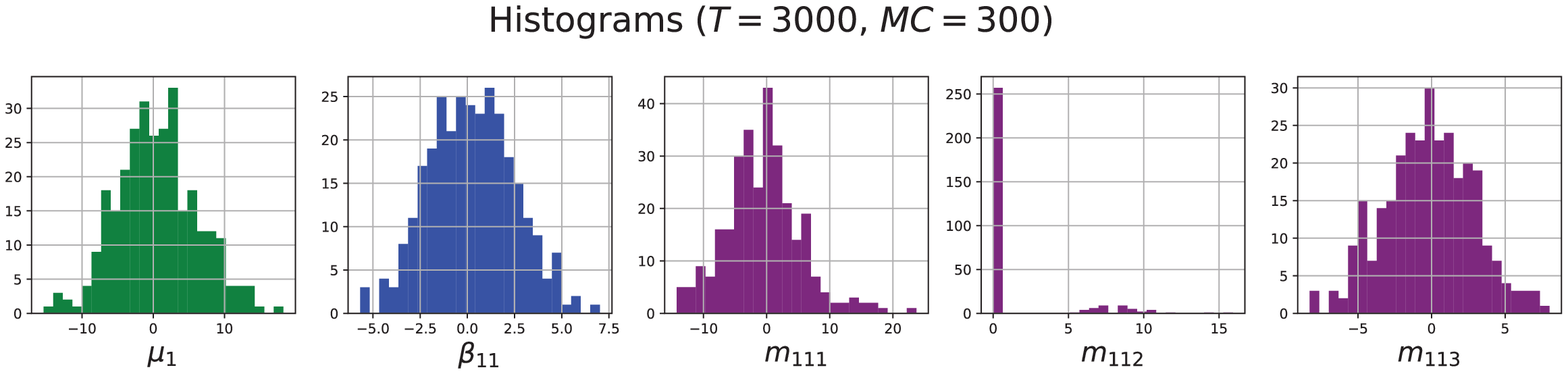}
	\caption{{\bf Histograms of the error of the P-O estimator $\sqrt{T} (\check{\vartheta}_T - \vartheta^*)$.}}
	\label{Ex and Sim Fig2-2-1}
\end{figure}

Table \ref{Ex and Sim Table2-2-2} shows the averages of squared errors of the QMLE $(\tilde{\vartheta}_T - \vartheta^*)^2$ and the P-O estimator $(\check{\vartheta}_T - \vartheta^*)^2$.
For non-zero parameters, owing to the asymptotic normality, both the QMLE and the P-O estimator have asymptotically the same level of variance.
For zero parameters, we can guess that the P-O estimator asymptotically has a smaller error than the QMLE, due to the accurate model selection.
However, when the observation time is small, the performance of the QMLE is better due to the miss model selection of the P-O estimator.

\begin{table}[htbp]
 \caption{{\bf Average of squared errors.}}
 \label{Ex and Sim Table2-2-2}
 \centering
  \begin{tabular}{ll | lllll}
   \hline
   $T$ & Method & $\mu$ & $\beta$ & $m_1$ & $m_2$ & $m_3$ \\ \hline\hline
   100 & QMLE & 3.22e-01 & 1.56e-00 & 2.59e-01 & {\bf 1.42e-01} & 8.29e-02 \\
   & P-OE & 7.52e-01 & 2.67e-00 & 3.27e-01 & {\bf 1.74e-01} & 1.10e-01 \\
   500 & QMLE & 6.09e-02 & 1.16e-02 & 6.11e-02 & {\bf 3.44e-02} & 1.98e-02 \\
   & P-OE & 6.40e-02 & 1.19e-02 & 7.97e-02 & {\bf 3.88e-02} & 2.65e-02 \\
   3000 & QMLE & 1.07e-02 & 1.74e-03 & 1.17e-02 & {\bf 4.74e-03} & 3.11e-03 \\
 	 & P-OE & 1.07e-02 & 1.73e-03 & 1.16e-02 & {\bf 3.61e-03} & 3.04e-03 \\
   \hline
  \end{tabular}
\end{table}


\section*{Acknowledgment}
I am deeply grateful to Professor Yoshida.
Without his guidance and help, I could not have completed this article.
This research was supported by the FMSP program of The University of Tokyo and Japan Science and Technology Agency CREST JPMJCR14D7.

\begin{appendix}
\section{Proofs}
\label{Appendix Proofs}
  \subsection{Proofs of Section \ref{P-O}}

  \begin{proof}[Proof of Theorem \ref{P-O Thm 1} (i)]
  By considering $0 \ge Q^{(q)}_T\big( \Stheta, \Snu \big) - Q^{(q)}_T\big( \theta^{0*}, \Snu \big)$ similarly to the proof of Theorem 1 in \cite{SuzukiYoshida2020}, we obtain the $\sqrt{T}$-consistency of $\Stheta$ from Assumption \ref{P-O Asm 1} (i) and (ii).
  Let
  \[
  	\Omega_{T, 1} = \Big\{ \omega \in \Omega \Big\arrowvert \exists j \in \calJ^1 \ s.t. \ \hat{\theta}^{0,(q)}_j = 0. \Big\},
  \]
  and
  \[
  	\Omega_{T, 2} = \Big\{ \omega \in \Omega \Big\arrowvert \exists j \in \calJ^0 \ s.t. \ \hat{\theta}^{0,(q)}_j \neq 0. \Big\}.
  \]
  Then we obtain $\big\{ \hat{\calJ}^0_T \neq \calJ^0 \big\} \subset \Omega_{T, 1} \cup \Omega_{T, 2}$.
  From the consistency of $\Stheta$,
  \begin{eqnarray}
  \label{Ap eq1-1}
  	P[\Omega_{T, 1}] &\le& P\Big[ \sqrt{T}\big\lvert \Stheta - \theta^{0*} \big\rvert \ge T^{\frac{1}{2}} c_0 \Big] \to 0
  \end{eqnarray}
  holds as $T \to \infty$, where $c_0 = \min_{j \in \calJ^1} \theta^{0*}_j >0$.
  On the other hand, to handle the case where the true value is at the boundary of the parameter space, we consider the following sets:
  \[
  	\Omega_{T, 2, 1} = \Big\{ \omega \in \Omega \Big\arrowvert \exists j \in \calJ^0 \ s.t. \ \left\lvert  \hat{\theta}^{0, (q)}_j - \theta^{0*}_j \right\lvert  \ge c_j \Big\},
  \]
  and
  \[
    \Omega_{T, 2, 2} = \Big\{ \omega \in \Omega \Big\arrowvert \exists j \in \calJ^0 \ s.t. \ \hat{\theta}^{0, (q)}_j \neq 0 \text{ and } \hat{\theta}^{0, (q)}_j \notin \partial \Theta^0_j \Big\},
  \]
  where $\Theta^0_j = \{ \theta^0_j \arrowvert ( \theta^0_1, \dots, \theta^0_{p_0} ) \in \Theta^0 \}$ and $c_j = \inf_{\theta^0_j \in \partial \Theta^0_j - \{0\} } \theta^0_j > 0$.
  Then, $\Omega_{T,2} \subset \Omega_{T, 2, 1} \cup \Omega_{T, 2, 2}$ holds.
  The consistency of $\Stheta$ immediately yields
  \begin{eqnarray}
  \label{Ap eq1-2}
  	P[\Omega_{T, 2, 1}] \to 0
  \end{eqnarray}
  as $T \to \infty$.
  On $\Omega_{T, 2, 2}$, $Q^{(q)}_T(\theta^0, \nu^0)$ is differentiable at $(\theta^0, \nu^0) = (\Stheta, \Snu)$ with respect to some $j$-th component.
  Thus, the same way as the proof of Theorem 2 in \cite{SuzukiYoshida2020} leads
  \begin{eqnarray}
  \label{Ap eq1-3}
  	P[\Omega_{T, 2, 2}] &\le& P\left[ \exists j \in \calJ^0 \ s.t. \ 2 \big\lvert \sqrt{T} \big( \hat{\theta}^{0,(q)}_j - \tilde{\theta}^0_j \big) \big\rvert \big\lvert \sqrt{T}\hat{\theta}^{0,(q)}_j \big\rvert^{1-q} \ge q T^{\frac{2-q}{2}} b_T \right] \nonumber\\
    &\to& 0.
  \end{eqnarray}
  From the equations (\ref{Ap eq1-1}), (\ref{Ap eq1-2}), and (\ref{Ap eq1-3}), we get the conclusion.
  \end{proof}

  We write $X_T \simle T^{-L}$ for a sequence $X_T$ and a positive constant $L$ if there exists a positive constant $C_L$ such that  $X_T \le C_LT^{-L}$ for all $T>0$.

  \begin{proof}[Proof of Theorem \ref{P-O Thm 1} (ii)]
  Same as the proof of Theorem 4 in \cite{SuzukiYoshida2020}, we have the $L^{\infty -}$-boundedness of $\sqrt{T} \big(\Stheta - \theta^{0*} \big)$, and $P[\Omega_{T, 1}]  \simle T^{-L}$, $P[\Omega_{T, 2, 2}]  \simle T^{-L}$.
  Moreover, we immediately obtain an inequality
  \[
  	P[\Omega_{T, 2, 1}] \le \frac{1}{(T^{\frac{1}{2}} \min_j c_j)^{2L}} E\left[ \left\lvert  \sqrt{T}\big(\Stheta - \theta^{0*} \big) \right\lvert ^{2L} \right] \ \simle\ T^{-L}
  \]
  for $c_j = \inf_{\theta^0_j \in \partial \Theta^0_j - \{0\} } \theta^0_j > 0$.
  Thus, $P\big[ \hat{\calJ}^0_T \neq \calJ^0 \big] \simle T^{-L}$ holds.
  \end{proof}

  Theorem \ref{P-O Thm 1} (iii) is obvious.
  The same way as the proof of Theorem 5 (b) in \cite{SuzukiYoshida2020} leads to Theorem \ref{P-O Thm 1} (iv).


  \subsection{Proofs of Section \ref{App to GEMHP}}
	\label{Appendix for Sec4}
  Let $E = \real_+ \times \real_+ \times \real^{p_1}$, and $\Dup$ be a set of functions $\phi \colon E \to \real$ such that:
  \begin{enumerate}
  	\itemi $\phi$ is of class $C^1$ on $(\real_+ - \{0\}) \times (\real_+ - \{0\}) \times \real^{p_1}$,
  	\itemii $\phi$ and $\lvert \Delta \phi \rvert$ are polynomial growth in $(u, v, w, u^{-1}1_{\{u \neq 0\}}, v^{-1}1_{\{v \neq 0\}})$ for $(u, v, w) \in E$,
  	\itemiii $\phi(0, v, w) = \phi(u, 0, w) = 0$.
  \end{enumerate}

  The sufficient conditions for the PLD are proposed by \cite{Clinet2021}.
  Here, modifying the conditions [A1]-[A3] in \cite{Clinet2021} to allow the model with nuisance parameter, we consider the following conditions.
	Here, we again call that $f(\vartheta)$ is of class $C^i(\Xi)$ for some $i \in \naturals$ if $f(\vartheta)$ is of class $C^i(\mathring{\Xi})$ and its derivatives admit continuous extensions on $\partial \Xi$.

	\begin{description}
  	\item[\textnormal{[A1]}] For any $i = 1, \dots, d$,
  		\begin{enumerate}
  			\itemi $\lambda^i_t(\vartheta)q_t(x, \vartheta)$ is a predictable on $\Omega \times \real_+ \times \bbX$ for any $\vartheta \in \Xi$,
  			\itemii $\vartheta \mapsto \lambda^i_t(\vartheta)q_t(x, \vartheta)$ is almost surely in $C^4(\Xi)$ for any $(t, x) \in \real_+ \times \bbX$,
  			\itemiii $\lambda^i_t(\vartheta)q_t(x, \vartheta)=0$ if and only if $\lambda^i_t(\vartheta^*)q_t(x, \vartheta^*) =0, dt\rho(dx)P(d\omega)\text{-}a.e.$ for any $\vartheta \in \Xi$ and $\nu^* \in \calN$.
  		\end{enumerate}

  	\item[\textnormal{[A2]}] For any $p>1$ and $i = 1, \dots, d$,
  		\begin{enumerate}
  			\itemi $\sup_{t \in \real_+} \sum_{n=0}^3 \left\| \sup_{\vartheta \in \Xi} \lvert \partial_{\theta}^n \lambda_t^i(\vartheta) \rvert \right\|_p < \infty$,
  			\itemii $\sup_{t \in \real_+}  \left\| \sup_{\vartheta \in \Xi} \lvert \lambda_t^i(\vartheta) \rvert^{-1} 1_{\{ \lambda_t^i(\vartheta) \neq 0 \}} \right\|_p < \infty,$
  			\itemiii {\small $\sup_{t \in \real_+} \sum_{n=0}^3 \int_{\bbX} E \left[ \sup_{\vartheta \in \Xi, \nu^* \in \calN} \lvert \partial_{\theta}^i \log q^i_t(x, \vartheta) \rvert^p q^i_t(x, \vartheta^*) \right]\rho(dx) < \infty$,}
  			\itemiv {\small $\sup_{t \in \real_+} \sum_{n=0}^3 \int_{\bbX} E \left[ \sup_{\vartheta \in \Xi, \nu^* \in \calN} \lvert\partial_{\theta}^i \log q^i_t(x, \vartheta) \rvert^{-p} q^i_t(x, \vartheta^*) \right]\rho(dx) < \infty$.}
  		\end{enumerate}

  	\item[\textnormal{[A3]}] There exist $\gamma \in (0, 1/2)$, $\pi_i \colon \Dup \times \Xi \times \calN \to \real$, and  $\chi_i \colon \{0,1,2\} \times \Xi \times \calN \to \real$ such that
  		\[
  			\sup_{\vartheta \in \Xi, \nu^* \in \calN} T^{\gamma} \left\| \frac{1}{T} \int_0^T \phi \left( \lambda^i_s(\vartheta^*), \lambda^i_s(\vartheta), \partial_{\theta} \lambda^i_s(\vartheta) \right) ds - \pi_i(\phi, \vartheta, \nu^*) \right\|_p \to 0
  		\]
  		as $T \to \infty$ for any $\phi \in \Dup$ and $p \ge 1$, and
  		{\small\begin{eqnarray*}
  			\sup_{\vartheta \in \Xi, \nu^* \in \calN} T^{\gamma} \Bigg\| \frac{1}{T} \int_{[0,T] \times \bbX} \partial^k_{\theta} \log q^i_s (x, \vartheta) q^i_s (x, \vartheta^*)\rho(dx)\lambda^i_s(\vartheta^*)ds - \chi_i(k, \vartheta, \nu^*) \Bigg\|_p \to 0
  		\end{eqnarray*}}
  		\hspace{-3pt}as $T \to \infty$ for any $k \in \{0,1,2\}$.
  \end{description}
	In [A1], which is a regularity condition to ensure the existence of the quasi log-likelihood process, we extended the differentiability of each function to the boundary of the parameter space.
	[A2] gives moment and smoothness conditions, and here, we consider the finiteness uniformly in $\nu^*$.
	[A3] is the condition for the ergodicity of $\lambda^i$ and $q^i$ uniformly in $(\vartheta, \nu^*) \in \Xi \times \calN$.
	These conditions are derived from the ergodicity of the GEMHP and the conditions [AH1]-[AH2].
  \begin{lemma}
  \label{Appendix Lem 1}
  	Under [L1]-[L2], [ND1]-[ND2], and [AH1]-[AH2], the conditions [A1]-[A3] hold.
  \end{lemma}
  \begin{proof}
  	The condition [A1] immediately follows from [AH1].
		Same as Lemma 6.6 in \cite{Clinet2021}, (i)-(ii) of [A2] hold.
		From [AH2] and Theorem \ref{GEMHP thm 1}, we readily obtain (iii)-(iv) of [A2].
  	With the help of uniformity in $\nu^*$ in [A2], the condition [A3] holds similar to the proof of Lemma 6.7 in \cite{Clinet2021}.
  \end{proof}

  We write $\Delta_T(\nu^*) = T^{-1/2} \partial_{\theta} l_T(\vartheta^*)$ and $\Gamma_T(\vartheta) = - T^{-1} \partial_{\theta}^2 l_T(\vartheta)$.
  Furthermore, we decompose $\bbY_T(\vartheta, \nu^*)$, $\Delta_T(\nu^*)$, and $\Gamma_T(\vartheta)$ as
  \begin{eqnarray*}
  	\bbY_T(\vartheta, \nu^*)
    &=& \bbY^{(1)}(\vartheta, \nu^*) + \bbY^{(2)}(\vartheta, \nu^*) \\
    &:=& \frac{1}{T} \big\{ l^{(1)}_T(\vartheta) - l^{(1)}_T(\vartheta^*) \big\} + \frac{1}{T} \big\{ l^{(2)}_T(\vartheta) - l^{(2)}_T(\vartheta^*) \big\},
  \end{eqnarray*}
  \[
  	\Delta_T(\nu^*) = \Delta^{(1)}_T(\nu^*) + \Delta^{(2)}_T(\nu^*) := T^{-1/2} \partial_{\theta} l^{(1)}_T(\vartheta^*) + T^{-1/2} \partial_{\theta} l^{(2)}_T(\vartheta^*),
  \]
  and
  \[
  	\Gamma_T(\vartheta) = \Gamma^{(1)}_T(\vartheta) + \Gamma^{(2)}_T(\vartheta) := - T^{-1} \partial_{\theta}^2 l^{(1)}_T(\vartheta) - T^{-1} \partial_{\theta}^2 l^{(2)}_T(\vartheta).
  \]
  For the proof of Theorem \ref{App to GEMHP Thm 1}, we prepare the following lemmas.

  \begin{lemma}
  \label{Appendix Lem 2}
  	Under [A1]-[A3], we have, for any $p >1$,
  	\begin{eqnarray}
  		\sup_{T \in \real_+} \left\| \sup_{\nu^* \in \calN} \big\lvert \Delta_T(\nu^*) \big\rvert \right\|_p < \infty. \nonumber
  	\end{eqnarray}
  \end{lemma}

  \begin{proof}
  	By Sobolev's inequality of Theorem 4.12 in \cite{AdamsFournier2003}, there exists a constant $A(\calN, p)$ such that
  	\begin{eqnarray}
  	\label{Appendix Eq 3-1}
  		&&  \left\| \sup_{\nu^* \in \calN} \big\lvert \Delta_T(\nu^*) \big\rvert \right\|_p^p \nonumber\\
  		&\le& A(\calN, p) \left\{ \int_{\calN} E \left[  \big\lvert \Delta_T(\nu^*) \big\rvert^p \right] d\nu^* + \int_{\calN} E \left[ \big\lvert \partial_{\nu^*} \Delta_T(\nu^*) \big\rvert^p \right] d\nu^* \right\} \nonumber\\
  		&\le& A(\calN, p) \diam(\calN) \left\{ \sup_{\nu^* \in \calN} E \left[  \big\lvert \Delta_T(\nu^*) \big\rvert^p \right]  + \sup_{\nu^* \in \calN} E \left[ \big\lvert \partial_{\nu^*} \Delta_T(\nu^*) \big\rvert^p \right] \right\}.
  	\end{eqnarray}
  	Since we immediately get that
  	\[
  		\Delta_T(\nu^*) = \frac{1}{\sqrt{T}} \sum_{i=1}^d \int_{[0,T] \times \bbX} \frac{\partial_{\vartheta} \big( \lambda^i_t(\vartheta^*) q^i_t(x, \vartheta^*) \big)}{ \lambda^i_t(\vartheta^*) q^i_t(x, \vartheta^*) } 1_{\{ \lambda^i_t(\vartheta^*) q^i_t(x, \vartheta^*) \neq 0 \}} \tilde{M}^i(dt, dx),
  	\]
  	where $\tilde{M}^i(dt, dx) = \bar{N}^i(dt, dx) -  \lambda^i_t(\vartheta^*) q^i_t(x, \vartheta^*) dt \rho(dx)$, the last term in (\ref{Appendix Eq 3-1}) is bounded uniformly in $T$ by Burkholder-Davis-Gundy inequality along with H\"older's inequality and [A2].
  	Thus, we have the conclusion.
  \end{proof}

  \begin{lemma}
  \label{Appendix Lem 3}
  	Under [A1]-[A3], we have, for any $p>1$,
  	\begin{eqnarray}
  		\sup_{T \in \real_+} T^{\gamma} \left\| \sup_{\nu^* \in \calN} \big\lvert \Gamma_T(\vartheta^*) - \Gamma(\nu^*) \big\rvert \right\|_p < \infty. \nonumber
  	\end{eqnarray}
  \end{lemma}

  \begin{proof}
  	By applying Sobolev's inequality of Theorem 4.12 in \cite{AdamsFournier2003}, we can choose $B(\calN, p)$ such that
  	\begin{eqnarray*}
  		&& \left( T^{\gamma} \left\| \sup_{\nu^* \in \calN} \big\lvert \Gamma^{(1)}_T(\vartheta^*) - \Gamma^{(1)}(\nu^*) \big\rvert \right\|_p \right)^p \\
  		&\le& B(\calN, p) T^{\gamma p} \Bigg\{ \int_{\calN} E \left[ \big\lvert \Gamma^{(1)}_T(\vartheta^*) - \Gamma^{(1)}(\nu^*) \big\rvert^p \right] d\nu^* \\
      &&+ \int_{\calN} E \left[ \big\lvert \partial_{\nu^*} \Gamma^{(1)}_T(\vartheta^*) - \partial_{\nu^*} \Gamma^{(1)}(\nu^*) \big\rvert^p \right] d\nu^* \Bigg\} \\
  		&\le& B(\calN, p) \diam(\calN) \Bigg\{ \sup_{\nu^* \in \calN} T^{\gamma p}E \left[ \big\lvert \Gamma^{(1)}_T(\vartheta^*) - \Gamma^{(1)}(\nu^*) \big\rvert^p \right] \\
      &&+ \sup_{\nu^* \in \calN} T^{\gamma p}E \left[ \big\lvert \partial_{\nu^*}\Gamma^{(1)}_T(\vartheta^*) - \partial_{\nu^*}\Gamma^{(1)}(\nu^*) \big\rvert^p \right] \Bigg\} \\
  		&\xrightarrow{T \to \infty}& 0,
  	\end{eqnarray*}
  	where the last convergence is a consequence of the condition [A3].
  	Similarly, we get
  	\[
  		T^{\gamma} \left\| \sup_{\nu^* \in \calN} \big\lvert \Gamma^{(2)}_T(\vartheta^*) - \Gamma^{(2)}(\nu^*) \big\rvert \right\|_p \to 0
  	\]
  	as $T \to \infty$, and thus we have the conclusion.
  \end{proof}

  \begin{proof}[Proof of Theorem \ref{App to GEMHP Thm 1}]
  	We only have to check that the conditions (A1''), (A4'), (A6), (B1), and (B2) in Theorem 3 (c) in  \cite{Yoshida2011} are satisfied.
  	Note that the nuisance parameter $\tau$ in \cite{Yoshida2011} is replaced by $\nu^*$ in our literature.
  	Set $\beta_1 = \gamma, \beta_2=\frac{1}{2} - \gamma, \rho = 2, \rho_2 \in (0, 2 \gamma), \alpha \in \left(0, \frac{\rho_2}{2}\right)$, and $\rho_1 \in \left(0, \min \left( 1, \frac{\alpha}{1-\alpha}, \frac{2\gamma}{1-\alpha}\right) \right)$ for satisfying (A4').
  	By Lemma \ref{Appendix Lem 1}, we have
  	\begin{eqnarray}
  	\label{App to GEMHP Thm 1-1}
  		T^{\gamma} \left\| \sup_{\vartheta \in \Xi,  \nu^* \in \calN} \big\lvert \bbY_T(\vartheta, \nu^*) - \bbY(\vartheta, \nu^*) \big\rvert \right\|_p \to^p 0,
  	\end{eqnarray}
  	as $T \to \infty$, for any $p>1$ in the same way as Theorem 2.2 in \cite{Clinet2021}.
  	Thus, (\ref{App to GEMHP Thm 1-1}) and Lemma \ref{Appendix Lem 2} lead to (A6).
  	Moreover, we have
  	\begin{eqnarray}
  	\label{App to GEMHP Thm 1-2}
  		\sup_{T \in \real_+} \left\| \frac{1}{T} \sup_{\vartheta \in \Xi, \nu^* \in \calN} \lvert \partial_{\theta}^3 l_T(\vartheta) \rvert \right\|_p < \infty
  	\end{eqnarray}
  	for any $p>1$ by Sobolev's inequality, H\"older's inequality, and [A2].
  	Now (A1'') is satisfied by (\ref{App to GEMHP Thm 1-2}) and Lemma \ref{Appendix Lem 3}.
  	Finally, from [AH3], we see that the conditions (B1) and (B2) follow immediately; see Remark \ref{App to GEMHP Rem 1}.
  \end{proof}

  \begin{proof}[Proof of Theorem \ref{App to GEMHP Thm 2}]
  	We immediately get the conclusion from Theorem \ref{App to GEMHP Thm 1}, Remark \ref{P-O Rem 1}, Remark \ref{P-O Rem 2} in this article, and Corollary 4.2 in \cite{Clinet2021}.
  \end{proof}

  \subsection{Proofs of Section \ref{Ex and Sim}}
  \begin{proof}[Proof of Proposition \ref{Ex and Sim Pro 1-1-1}]
  	Since there are no marks, we can assume that $\bbX = \real$, $Q_j(x, dy) = \delta_0(dx)$, and $g_{ij}(x)=1$ for all $i, j = 1, \dots, d$.
  	Then, the conditions [L1], [ND1], and [ND2] hold, for example, for $f_X(x)=\lvert x \rvert$ and $u_X(x) = \big( \sum_{i=1}^d \mu^*_i \big)\lvert x \rvert$.
  	The condition [L2] holds by the Perron-Frobenius' theorem and the assumption that the spectral radius of $\Phi$ is less than $1$.
  	Thus, the multivariate exponential Hawkes process with sparse structure has geometric ergodicity.

  	Since each $\mu_i$ takes a positive value, the conditions [AH1] and [AH2] hold.
  	Finally, the condition [AH3] is satisfied in the same way as Lemma A.7 in \cite{ClinetYoshida2017}.
  	To obtain [AH3], we prove that $\bby = 0$ if there exists a vector $\bby$ such that $\inf_{\nu^* \in \calN} \bby^T \Gamma(\nu^*)\bby = 0$.
  	In this proof, we use Assumption \ref{Ex and Sim Asm 1-1-1} (i) to take positive $\nu^*$ realizing the infimum.
  \end{proof}

  \begin{proof}[Proof of Proposition \ref{Ex and Sim Pro 2-1-1}]
  	Let
  	\[
  		g_{ij}(x) = \left\{\begin{array}{ll}
  									1 & \text{if $m^*_{ij,l} = 0$ for all $l = 1, \dots, d'$,}\\
  									\sum_{l=1}^{d'} m^*_{ijl}x_{l} & \text{otherwise,}
  								\end{array}\right.
    \]
    and
    \[
  		\alpha^*_{ij} = \left\{\begin{array}{ll}
  									0 & \text{if $m^*_{ijl} = 0$ for all $l = 1, \dots, d'$,}\\
  									1 & \text{otherwise,}
  								\end{array}\right.
  	\]
  	and then we can write
  	\[
  		\lambda^i_t(\vartheta^*) = \mu^*_i + \sum_{j=1}^d \int_{[0,t) \times \mathbb{X}} \alpha^*_{ij}e^{-(t-s)\beta^*_{ij}}g_{ij}(x) \bar{N}^j(ds, dx), \quad i = 1, \dots, d.
  	\]
  	We check the conditions [L1]-[L2] and [ND1]-[ND2] for this GEMHP.
  	First, the condition [ND1] obviously holds.
  	Let $f_X(x) = \lvert x \rvert$ and $u_X(x) = -\sum_{i=1}^d \mu^*_i \int_{\bbX} \lvert y \rvert - \lvert x \rvert Q_i(x, dy)$ for $x \in \real^{d'}$, where we extend the first domain of $Q_i(x, dy)$ as $Q_i(x, dy)=\delta_{(1,0,\dots,0)}(dy)$ for $x \notin \bbX$.
  	Then, we can easily check the condition [L1].
  	The conditions [L2] and [ND2] are assumed in Assumption \ref{Ex and Sim Asm 2-1-1} (ii) and (iii).
  	Thus, this model has geometric ergodicity.

  	Now, we consider the conditions [AH1]-[AH3].
  	We define $g_{ij}(x, \vartheta) = \sum_{l=1}^{d'} m_{ijl}x_{l}$, and then we can write
  	\[
  		\lambda^i_t(\vartheta) = \mu_i + \sum_{j=1}^d \int_{[0,t) \times \mathbb{X}} e^{-(t-s)\beta_{ij}}g_{ij}(x, \vartheta) \bar{N}^j(ds, dx), \quad i = 1,\dots d.
  	\]
  	[AH1] obviously holds, and [AH2] follows from Assumption \ref{Ex and Sim Asm 2-1-1} (v).
  	We write $\vartheta = (\theta, \nu)$ as in Section \ref{App to GEMHP}.
  	Finally, for the sake of the condition [AH3], we only have to show that (i) $\bbY(\vartheta, \nu^*) < 0$ for any $\theta \in \Theta-\{\theta^*\}$ and $\nu, \nu^* \in \calN$, and (ii) $\Gamma(\nu^*)$ is positive definite uniformly in $\nu^* \in \calN$, see Remark \ref{App to GEMHP Rem 1}.
  	Let there be $\theta \in \Theta-\{\theta^*\}$ and $\nu, \nu^* \in \calN$ such that $\bbY(\vartheta, \nu^*) = 0$.
  	We have
  	\begin{eqnarray*}
  		0 = -\bbY(\vartheta, \nu^*)
  		&=& \sum_{i=1}^d \Bigg\{ E\bigg[ \lambda'^i_t(\vartheta) - \lambda'^i_t(\vartheta^*) - \log\bigg( \frac{\lambda'^i_t(\vartheta)}{\lambda'^i_t(\vartheta^*)}\bigg) \lambda'^i_t(\vartheta^*) \bigg] \\
      &&+ E\bigg[ \int_{\bbX} \log\bigg( \frac{q'^i_t(x, \theta_M^*)}{q'^i_t(x, \theta_M)}\bigg) q'^i_t(x, \theta_M^*) \lambda'^i_t(\vartheta^*) dx \bigg] \Bigg\}.
  	\end{eqnarray*}
  	Since each term on the right-hand side is non-negative, we have
  	\[
  		\lambda'^i_t(\vartheta) - \lambda'^i_t(\vartheta^*) - \log\bigg( \frac{\lambda'^i_t(\vartheta)}{\lambda'^i_t(\vartheta^*)}\bigg) \lambda'^i_t(\vartheta^*) = 0, \quad a.s.
  	\]
  	and
  	\[
  		\log\bigg( \frac{q'^i_t(x, \theta_M^*)}{q'^i_t(x, \theta_M)}\bigg) q'^i_t(x, \theta_M^*) \lambda'^i_t(\vartheta^*) = 0, \quad dxP(d\omega)\text{-}a.e.
  	\]
  	Thus, we obtain $\lambda'^i_t(\vartheta) = \lambda'^i_t(\vartheta^*) \ a.s.$ and $q'^i_t(x, \theta_M) = q'^i_t(x, \theta_M^*) \ a.e.$
  	Then, we get $\theta_M = \theta_M^*$ by Assumption \ref{Ex and Sim Asm 2-1-1} (v-a).
  	On the other hand, we have
  	\begin{eqnarray}
  		\label{Ap eq3-1}
  		\mu_i^* - \mu_i &=& \sum_{j=1}^d \int_{(-\infty, t) \times \bbX} e^{-\beta_{ij}(t-s)}\bigg(\sum_{l=1}^{d'}m_{ijl} x_l\bigg) \nonumber\\
      &&- e^{-\beta_{ij}^*(t-s)}\bigg(\sum_{l=1}^{d'}m_{ijl}^* x_l\bigg) \bar{N}'^j(ds, dx), \quad a.s.,
  	\end{eqnarray}
  	for any $i = 1, \dots, d$.
  	Since the left-hand side is constant, the right-hand side must only have jumps of size zero.
  	Then, by Assumption \ref{Ex and Sim Asm 2-1-1} (iv), we get $m_{ijl}^* = m_{ijl}$ for all $i, j$, and $l$.
  	By taking the derivative with respect to $t$, we have
  	\begin{eqnarray*}
  		0 &=& \sum_{j=1}^d \int_{(-\infty, t) \times \bbX} \beta_{ij}e^{-\beta_{ij}(t-s)}\bigg(\sum_{l=1}^{d'}m_{ijl} x_l\bigg) \\
      &&- \beta_{ij}^*e^{-\beta_{ij}^*(t-s)}\bigg(\sum_{l=1}^{d'}m_{ijl}^* x_l\bigg) \bar{N}'^j(ds, dx), \quad a.s.,
  	\end{eqnarray*}
  	and thus, we get $\beta_{ij}^* = \beta_{ij}$ for all $i$ and $j$ such that $m_{ijl}^*>0$ for some $l$ by Assumption \ref{Ex and Sim Asm 2-1-1} (iv).
  	Finally, the right-hand side of (\ref{Ap eq3-1}) becomes zero, and $\mu_i^* = \mu_i$ holds.
  	Therefore, $\theta = \theta^*$ holds and contradicts $\theta \in \Theta-\{\theta^*\}$.

  	Next, we assume that there exists $\bby \in \real^{p_0+p_1}$ such that $\inf_{\nu^* \in \calN} \bby^T \Gamma(\nu^*)\bby = 0$.
  	We write $\bby = (\bby_{\lambda}, \bby_M)$, where $\bby_{\lambda}$ and $\bby_M$ are related to $\theta_{\lambda}$ and $\theta_M$, respectively, and $\theta_{\lambda}$ denotes non-nuisance parameters $\mu, \beta, m$.
  	From Assumption \ref{Ex and Sim Asm 2-1-1} (v-c), we have
  	\begin{eqnarray*}
  		0 = \inf_{\nu^* \in \calN}  \bby^T \Gamma(\nu^*) \bby
  		&=& \inf_{\nu^* \in \calN} \sum_{i=1}^d \Bigg\{ E\Bigg[ \frac{\big( \bby_{\lambda}^T \partial_{\theta_{\lambda}}\lambda'^i_t(\vartheta^*) \big)^2}{\lambda'^i_t(\vartheta^*)} \Bigg] \\
      &&+ E\Bigg[ \int_{\bbX} \big(\bby_M^T \partial_{\theta_M} \log q'^i_t(x,\theta_M^*) \big)^2q'^i_t(x, \theta_M^*) \lambda'^i_t(\vartheta^*) dx \Bigg]\Bigg\},
  	\end{eqnarray*}
  	and then it is necessary that almost surely
  	\begin{eqnarray}
  		\label{Ap eq3-2}
  		\inf_{\nu^* \in \calN} \big( \bby_{\lambda}^T \partial_{\theta_{\lambda}}\lambda'^i_t(\vartheta^*) \big)^2 = 0
  	\end{eqnarray}
  	and
  	\begin{eqnarray}
  		\label{Ap eq3-3}
  		\bby_M^T \partial_{\theta_M} \log q'^i_t(x,\theta_M^*) = 0
  	\end{eqnarray}
  	hold for any $t \ge 0$ and $i = 1, \dots, d$.
  	Then, (\ref{Ap eq3-3}) and Assumption \ref{Ex and Sim Asm 2-1-1} (v-b) lead to $\bby_M = 0$.
  	On the other hand, from Assumption \ref{Ex and Sim Asm 2-1-1} (i), the infimum in (\ref{Ap eq3-2}) for each path is realized by some positive $\beta^*_{ij}$, and we again write this point as $\beta^*_{ij}$.
  	Then, for $i= 1, \dots, d$,
  	\begin{eqnarray*}
  		0 &=& \bby_{\lambda}^T \partial_{\theta_{\lambda}}\lambda'^i_t(\vartheta^*) \\
  		&=& y_{\mu_i} + \sum_{j=1}^d \int_{(-\infty, t) \times \bbX} \Bigg\{ -(t-s)\bigg(\sum_{l=1}^{d'}m_{ijl}^*x_l\bigg)y_{\beta_{ij}} + \sum_{l=1}^{d'}x_ly_{m_{ijl}} \Bigg\} \\
      &&e^{-\beta_{ij}^*(t-s)} \bar{N}'^j(ds, dx), \quad a.s.,
  	\end{eqnarray*}
  	where $y_{\mu_i}, y_{\beta_{ij}}$, and $y_{m_{ijl}}$ are components of $\bby$ corresponding to non-nuisance parameters $\mu_i, \beta_{ij}$, and $y_{m_{ijl}}$, respectively.
  	Differentiating both sides $n$-times with respect to $t$, we almost surely get
  	\begin{eqnarray}
  		\label{Ap eq3-4}
  		0 &=& \sum_{j=1}^d \int_{(-\infty, t) \times \bbX} \Bigg\{ \big\{(t-s)\beta_{ij}^* -n\big\}\bigg(\sum_{l=1}^{d'}m_{ijl}^*x_l\bigg)y_{\beta_{ij}} -\beta_{ij}^*\sum_{l=1}^{d'}x_ly_{m_{ijl}} \Bigg\} \nonumber\\
      &&(-\beta_{ij}^*)^{n-1} e^{-\beta_{ij}^*(t-s)} \bar{N}'^j(ds, dx).
  	\end{eqnarray}
		Let $\Omega_{j'} = \big\{\omega \in \Omega \big\arrowvert T_1^{j'} < 1 \text{ and } T_1^k > 1 \text{ for $k \neq j'$.} \big\}$, where $T_1^j$ represents the first jump time of the counting process $N'^j_t = \bar{N}'^j\big([0,t] \times \bbX \big)$.
		We easily see $P[\Omega_{j'}] > 0$.
		By taking the limit $t \searrow T_1^{j'}$ in (\ref{Ap eq3-4}) on the set $\Omega_{j'}$, we obtain
		\[
			0 =  \bigg(\sum_{l=1}^{d'}m_{ijl}^*X^l_{T_1^{j'}}\bigg)y_{\beta_{ij}} + \frac{\beta_{ij}^*}{n}\sum_{l=1}^{d'}X^l_{T_1^{j'}}y_{m_{ijl}} \to \bigg(\sum_{l=1}^{d'}m_{ijl}^*X^l_{T_1^{j'}}\bigg)y_{\beta_{ij}} \quad \text{ as $n\to\infty$},
		\]
  	and thus we get $y_{\beta_{ij}} = 0$ since $m_{ijl}^*>0$ holds for some $l$ by the definition of $y_{\beta_{ij}}$.
  	Then, we immediately obtain $y_{m_{ijl}}=0$ by Assumption \ref{Ex and Sim Asm 2-1-1} (iv), and thus $y_{\mu_i}=0$ also holds.
  	Now $\bby = 0$, and we get the conclusion.
  \end{proof}


\section{Additional numerical experiments}
\label{Appendix Sim}
\subsection{Scenario with no zero coefficients}
We considered scenarios with no zero coefficients to see if the P-O estimator and QMLE perform similarly.
Here, we deal with the Hawkes process marked with "Topic" introduced in Section \ref{Ex and Sim 2}.
Let $\bar{N}$ be the $1$-dimensional GEMHP whose intensity is
\begin{eqnarray}
 	\lambda_t(\vartheta^*) &=& \mu + \int_{[0,t) \times \mathbb{X}} e^{-\beta(t-s)} \big(m_1x_1 + m_2x_2 + m_3x_3 \big)\bar{N}(ds, dx) \bigg\arrowvert_{\vartheta = \vartheta^*}\nonumber\\
 	&=& 1.2 + \int_{[0,t) \times \mathbb{X}} e^{-0.5(t-s)} \big(0.5x_1 + 0.3x_2 + 0.4x_3 \big)\bar{N}(ds, dx),\nonumber
\end{eqnarray}
where its marks independently and identically follow the $3$-dimensional Dirichlet distribution with a parameter $\alpha = (2, 2, 5)$.
Same as Subsection \ref{Ex and Sim 2-2}, we estimate the parameters $m_1, m_2, m_3, \mu, \beta$ and assume that only parameters $m$'s can take the zero value, i.e., we set $\theta^0 = (m_1, m_2, m_3)$ and $\theta^1 = (\mu, \beta)$.
We also set the hyperparameters of the P-O estimator to be $q=1.0, \gamma=2.0, a=0.5$, the observation times $T = 100, 500, 3000$, and the number of the Monte Carlo simulation $MC = 300$.

Table \ref{Appendix Table2-1} shows the fraction of trials in which the parameters $m_1, m_2$, and $m_3$ are estimated to be completely zero.
We see that both methods asymptotically make correct model selections, but that the probability of incorrectly estimating zero is higher for the P-O estimator.
\begin{table}[htbp]
	\centering
	\caption{{\bf Percentage of estimated to be zero.}}
	\label{Appendix Table2-1}
  \scalebox{0.9}{
	\begin{tabular}{| l | l | l | l | l | l |} \hline
		\multicolumn{6}{| l |}{QMLE: $T=100$} \\ \hline
		$m_1$ & 22.0\% & $m_2$ & 35.3\% & $m_3$ & 16.7\% \\ \hline
		\multicolumn{6}{| l |}{QMLE: $T=500$} \\ \hline
		$m_1$ & 3.67\% & $m_2$ & 12.0\% & $m_3$ & 0.33\% \\ \hline
		\multicolumn{6}{| l |}{QMLE: $T=3000$} \\ \hline
		$m_1$ & 0.00\% & $m_2$ & 1.67\% & $m_3$ & 0.00\% \\ \hline
	\end{tabular}
	\begin{tabular}{| l | l | l | l | l | l |} \hline
		\multicolumn{6}{| l |}{P-OE: $T=100$} \\ \hline
		$m_1$ & 35.3\% & $m_2$ & 57.3\% & $m_3$ & 44.0\% \\ \hline
		\multicolumn{6}{| l |}{P-OE: $T=500$} \\ \hline
		$m_1$ & 11.0\% & $m_2$ & 36.3\% & $m_3$ & 6.00\% \\ \hline
		\multicolumn{6}{| l |}{P-OE: $T=3000$} \\ \hline
		$m_1$ & 0.00\% & $m_2$ & 7.33\% & $m_3$ & 0.00\% \\ \hline
	\end{tabular}
  }
\end{table}

Table \ref{Appendix Table2-2} shows the averages of squared errors of the QMLE $(\tilde{\vartheta}_T - \vartheta^*)^2$ and the P-O estimator $(\check{\vartheta}_T - \vartheta^*)^2$.
When the observation time is small, we see that the QMLE perform better than the P-O estimator due to the miss model selection of the P-O estimator.
However, the difference becomes smaller as the observation time is longer.
We note that both the QMLE and the P-O estimator have asymptotic normality with the same variance.

\begin{table}[htbp]
 \caption{{\bf Average of squared errors.}}
 \label{Appendix Table2-2}
 \centering
  \begin{tabular}{ll | lllll}
   \hline
   $T$ & Method & $\mu$ & $\beta$ & $m_1$ & $m_1$ & $m_3$ \\ \hline\hline
   100 & QMLE & 3.69e-01 & 4.20e-02 & 2.81e-01 & 1.97e-01 & 7.09e-02 \\
   & P-OE & 7.58e-01 & 8.85e-01 & 5.45e-01 & 2.40e-01 & 9.63e-02 \\
   500 & QMLE & 7.41e-02 & 5.30e-03 & 8.55e-02 & 5.39e-02 & 2.01e-02 \\
   & P-OE & 7.54e-02 & 5.43e-03 & 1.00e-01 & 6.80e-02 & 2.58e-02 \\
   3000 & QMLE & 1.13e-02 & 7.02e-04 & 1.32e-02 & 1.28e-02 & 3.36e-03 \\
 	 & P-OE & 1.12e-02 & 7.00e-04 & 1.39e-02 & 1.53e-02 & 3.51e-03 \\
   \hline
  \end{tabular}
\end{table}


\subsection{Comparison with previous studies}
In this subsection, we compare the performance of the P-O estimator with the mixed method of Lasso and nuclear regularization, introduced in \cite{ZhouEtal2013}, and the elastic net.
These classical methods are implemented in tick library\footnote{The documentation is available here \url{https://x-datainitiative.github.io/tick/}.} in Python3, see \cite{BacryEtal2018}, and work only for an exponential Hawkes model whose decay parameter is given.

Here, we consider the $4$-dimensional exponential Hawkes process $N_t = (N^1_t, \dots, N^4_t)$, see Eq. (\ref{Ex and Sim Eq 1-1-1}), whose intensity with the following parameters:
\[
	\mu^* = (0.05, 0.05, 0.05), \quad
	\alpha^* = \left(
		\begin{array}{cccc}
            0.15 & 0 & 0 & 0\\
            0.15 & 0 & 0 & 0\\
            0 & 0.1 & 0.1 & 0.1\\
            0 & 0.1 & 0.1 & 0.1
		\end{array}
		\right),
\]
and the decay parameter is given by $\beta^*_{ij} = 1.0$ for all $i, j = 1, \dots, 4$.
We only estimate the parameters $\vartheta=(\mu, \alpha)$ and assume that only parameters $\alpha$ can take the zero value, i.e., we set $\theta^0 = \alpha$ and $\theta^1 = \mu$.

The mixed method of \cite{ZhouEtal2013} is defined by
\begin{eqnarray}
\label{Appendix Eq 2-2-1}
	\tilde{\vartheta}_T  \in argmin_{\vartheta \in \Xi} \Big[ -l_T(\vartheta) + C_m \big\{ (1-\rho_m) \| \alpha \|_* + \rho_m \| \alpha \|_1 \big\} \Big],
\end{eqnarray}
where $C_m$ and $\rho_m$ are hyperparameters, $\| \cdot \|_*$ is the nuclear norm of a matrix, which is defined to be the sum of its singular value, $\| \cdot \|_1$ is the $L^1$ norm, and $l_T$ is the log-likelihood process of the Hawkes process.
On the other hand, the elastic net is given by
\begin{eqnarray}
\label{Appendix Eq 2-2-2}
	\tilde{\vartheta}_T \in argmin_{\vartheta \in \Xi} \Big[ R_T(\vartheta) + C_e \big\{ (1-\rho_e) \| \alpha \|_1 + \rho_e \| \alpha \|_2 \big\} \Big],
\end{eqnarray}
where $C_e$ and $\rho_e$ are hyperparameters, $\| \cdot \|_2$ is the $L^2$ norm, and $R_T$ is the least-squares function for the Hawkes process, that is,
\[
    R_T(\vartheta) = \frac{1}{T} \sum_{i=1}^d \left\{ \int_0^T \big(\lambda^i_t(\vartheta) \big)^2 dt - 2 \int_0^T \lambda^i_t(\vartheta) N^i(dt) \right\}.
\]

We set the hyperparameters in Eqs. (\ref{Appendix Eq 2-2-1}) and (\ref{Appendix Eq 2-2-2}) to be $C_m = C_e = 1000$, $\rho_m = 0.5$, and $\rho_e = 0.95$, and the hyperparameters of the P-O estimator to be $q=1.0, \gamma=1.0, a=0.5$.
Let the observation times $T = 3000$ and the number of the Monte Carlo simulation $MC = 300$.

Table \ref{Appendix Table2-2-1} shows the fraction of trials in which the parameter $\alpha_{ij}$'s are estimated to be completely zero.
Here, we regarded estimated values less than 1.0e-8 as zero by taking into account the numerical error in the tick library.
We can see that the variable selection is performed more accurately by the P-O estimator than by the other methods.

\begin{table}[htbp]
  \centering
	\caption{{\bf Percentage of estimated to be zero.}}
	\label{Appendix Table2-2-1}
	\begin{tabular}{| l | l | l | l | l | l | l | l |} \hline
		\multicolumn{8}{| l |}{Mixed Method in \cite{ZhouEtal2013}} \\ \hline
		$\alpha_{11}$ & 0.00\% & $\alpha_{12}$ & {\bf 40.3\%} & $\alpha_{13}$ &  {\bf 40.3\%} & $\alpha_{14}$ & {\bf 40.0\%} \\ \hline
		$\alpha_{21}$ & 0.00\% & $\alpha_{22}$ & {\bf 43.7\%}  & $\alpha_{23}$ & {\bf 39.7\%} & $\alpha_{24}$ & {\bf 40.3\%} \\ \hline
		$\alpha_{31}$ & {\bf 32.0\%} & $\alpha_{32}$ & 0.00\% & $\alpha_{33}$ & 0.00\% & $\alpha_{34}$ & 0.00\%\\ \hline
		$\alpha_{41}$ & {\bf 41.7\%} & $\alpha_{42}$ & 0.00\% & $\alpha_{43}$ & 0.00\% & $\alpha_{45}$ & 0.00\% \\ \hline\hline
		\multicolumn{8}{| l |}{Elastic Net} \\ \hline
		$\alpha_{11}$ & 0.00\% & $\alpha_{12}$ & {\bf 56.7\%} & $\alpha_{13}$ &  {\bf 59.7\%} & $\alpha_{14}$ & {\bf 61.0\%} \\ \hline
		$\alpha_{21}$ & 0.00\% & $\alpha_{22}$ & {\bf 60.3\%}  & $\alpha_{23}$ & {\bf 60.7\%} & $\alpha_{24}$ & {\bf 60.0\%} \\ \hline
		$\alpha_{31}$ & {\bf 55.0\%} & $\alpha_{32}$ & 0.67\% & $\alpha_{33}$ & 0.00\% & $\alpha_{34}$ & 0.00\%\\ \hline
		$\alpha_{41}$ & {\bf 57.7\%} & $\alpha_{42}$ & 0.00\% & $\alpha_{43}$ & 0.00\% & $\alpha_{45}$ & 0.00\% \\ \hline\hline
		\multicolumn{8}{| l |}{P-O Estimator} \\ \hline
		$\alpha_{11}$ & 0.00\% & $\alpha_{12}$ & {\bf 93.3\%} & $\alpha_{13}$ &  {\bf 92.3\%} & $\alpha_{14}$ & {\bf 93.0\%} \\ \hline
		$\alpha_{21}$ & 0.33\% & $\alpha_{22}$ & {\bf 92.7\%}  & $\alpha_{23}$ & {\bf 95.3\%} & $\alpha_{24}$ & {\bf 95.7\%} \\ \hline
		$\alpha_{31}$ & {\bf 90.3\%} & $\alpha_{32}$ & 4.33\% & $\alpha_{33}$ & 1.00\% & $\alpha_{34}$ & 1.00\%\\ \hline
		$\alpha_{41}$ & {\bf 90.3\%} & $\alpha_{42}$ & 3.00\% & $\alpha_{43}$ & 2.00\% & $\alpha_{45}$ & 2.00\% \\ \hline
	\end{tabular}
\end{table}

Table \ref{Appendix Table2-2-2} shows the averages of squared errors of each method.
For non-zero parameters, each method has almost the same level of variance.
For zero parameters, we get a smaller error by the P-O estimator than by the other methods due to the accurate model selection.

\begin{table}[htbp]
	{\footnotesize
 	\caption{{\bf Average of squared errors.}}
 	\label{Appendix Table2-2-2}
  \begin{tabular}{l | llll ll}
		\hline
		Method & $\mu_1$ & $\mu_2$ & $\mu_3$ & $\mu_4$ & $\alpha_{11}$ & $\alpha_{12}$ \\ \hline\hline
   	Mixed Method & 2.71e-05 & 2.17e-05 & 2.69e-05 & 2.52e-05 & 1.38e-03 & {\bf 2.79e-04} \\
   	Elastic Net & 2.71e-05 & 2.14e-05 & 3.06e-05 & 2.76e-05 & 1.73e-03 & {\bf 2.71e-04} \\
		P-OE & 2.43e-05 & 2.06e-05 & 2.72e-05 & 2.55e-05 & 1.38e-03 & {\bf 1.70e-04} \\
   	\hline
 	\end{tabular}
 	\vspace{2mm}

 	\begin{tabular}{l | ll llll}
	 	\hline
		Method & $\alpha_{13}$ & $\alpha_{14}$ & $\alpha_{21}$ & $\alpha_{22}$ & $\alpha_{23}$ & $\alpha_{24}$ \\ \hline\hline
   	Mixed Method & {\bf 2.67e-04} & {\bf 2.77e-04} & 1.47e-03 & {\bf 3.39e-04} & {\bf 2.15e-04} & {\bf 1.78e-04} \\
   	Elastic Net & {\bf 2.32e-04} & {\bf 2.56e-04} & 1.81e-03 & {\bf 2.38e-04} & {\bf 1.88e-04} & {\bf 1.45e-04} \\
		P-OE & {\bf 1.71e-04} & {\bf 1.81e-04} & 1.51e-03 & {\bf 2.26e-04} & {\bf 1.21e-04} & {\bf 8.14e-05} \\
   	\hline
 	\end{tabular}
	\vspace{2mm}

 	\begin{tabular}{l | llll}
	 	\hline
		Method & $\alpha_{31}$ & $\alpha_{32}$ & $\alpha_{33}$ & $\alpha_{34}$ \\ \hline\hline
   	Mixed Method & {\bf 3.98e-04} & 1.36e-03 & 1.09e-03 & 1.16e-03 \\
   	Elastic Net & {\bf 4.14e-04} & 1.71e-03 & 1.24e-03 & 1.27e-03 \\
		P-OE & {\bf 2.76e-04} & 1.52e-03 & 1.14e-03 & 1.21e-03 \\
   	\hline
 	\end{tabular}
	\vspace{2mm}

 	\begin{tabular}{l | llll}
	 	\hline
		Method & $\alpha_{41}$ & $\alpha_{42}$ & $\alpha_{43}$ & $\alpha_{44}$ \\ \hline\hline
   	Mixed Method & {\bf 4.02e-04} & 1.31e-03 & 1.21e-03 & 1.09e-03 \\
   	Elastic Net & {\bf 4.15e-04} & 1.46e-03 & 1.37e-03 & 1.29e-03 \\
		P-OE & {\bf 3.00e-04} & 1.45e-03 & 1.31e-03 & 1.17e-03 \\
   	\hline
 	\end{tabular}}
\end{table}



\end{appendix}

\bibliography{ref.bib}


\end{document}